\documentclass{article}

\usepackage{amsmath,amsfonts,amssymb,amsthm}
\usepackage{booktabs}
\usepackage{graphicx}
\usepackage[margin=1.3in]{geometry}

\usepackage{siunitx}
\usepackage{paralist}
\usepackage{url}
\usepackage[round]{natbib}

\newtheorem{theorem}{Theorem}
\newtheorem{lemma}{Lemma}
\newtheorem{proposition}{Proposition}

\renewcommand{\le}{\leqslant}
\renewcommand{\ge}{\geqslant}
\newcommand{\real}{\mathbb{R}}
\newcommand{\phz}{\phantom{0}}
\newcommand{\wh}{\widehat}
\newcommand{\bca}{BC$_a$}
\newcommand{\bsv}{\boldsymbol{v}}
\newcommand{\dbeta}{\mathrm{Beta}}
\newcommand{\dbin}{\mathrm{Bin}}
\newcommand{\dexp}{\mathrm{Exp}}

\newcommand{\dnorm}{\mathcal{N}}
\newcommand{\dpois}{\mathrm{Pois}}
\newcommand{\dunif}{\mathbb{U}}

\newcommand{\e}{\mathbb{E}}
\newcommand{\var}{\mathrm{var}}

\renewcommand{\skew}{\mathrm{skew}}

\newcommand{\eff}{\mathrm{eff}}
\newcommand{\sign}{\mathrm{sign}}

\newcommand{\simiid}{\stackrel{\mathrm{iid}}\sim}
\newcommand{\dotsim}{\stackrel{\cdot}{\sim}}

\newcommand{\rd}{\,\mathrm{d}}
\newcommand{\giv}{\!\mid\!}
\newcommand{\cf}{\mathcal{F}}

\newcommand{\olvv}{\overline{v^2}}
\newcommand{\olvk}{\overline{v^k}}
\newcommand{\olvvv}{\overline{v^3}}

\newcommand{\rmsl}{\mathrm{RMSL}}

\title{Better bootstrap $t$ confidence intervals for the mean}
\author{Art B. Owen\\Stanford University}
\date{August 2025}
\begin{document}
\maketitle

\begin{abstract}
This article explores combinations of weighted bootstraps,
like the Bayesian bootstrap, with the bootstrap $t$ method
for setting approximate confidence intervals for the mean
of a random variable in small samples.
For this problem the usual bootstrap $t$ has good
coverage but provides intervals with long and highly
variable lengths.  Those intervals can have infinite length
not just for tiny $n$, when the data have a discrete distribution.
The \bca\ bootstrap produces shorter intervals but
tends to severely under-cover the mean.
Bootstrapping the studentized mean with weights from a Beta$(1/2,3/2)$ distribution 
is shown to attain second order accuracy. It never yields infinite length
intervals and the mean square bootstrap $t$ statistic is finite when
there are at least three distinct values in the data, or two distinct
values appearing at least three times each.
In a range of small sample settings, the beta bootstrap $t$
intervals have closer to nominal coverage than the \bca\
and shorter length than the multinomial bootstrap $t$.
The paper includes a lengthy discussion of the difficulties in
constructing a utility function to evaluate nonparametric approximate
confidence intervals.
\end{abstract}

\section{Introduction}

This article is about nonparametric approximate confidence intervals
(ACIs) for the mean of a real  random variable.  There are two leading bootstrap
solutions for this problem. One is the bootstrap $t$ method
and the other is the bias corrected accelerated (BCa) method.
These are both well understood theoretically in the limit
as the sample size $n$ of an IID sample diverges to infinity.
Under mild conditions on the sampling distribution, both
are second order accurate with both one and two-sided
coverage errors decaying as $O(1/n)$.  Alternative methods that
are not second order accurate typically have one-sided
coverage errors of $O(1/\sqrt{n})$ with two-sided
errors of $O(1/n)$ \citep{hall:1988}.

For smaller values of $n$ there is a nasty tradeoff.
The bootstrap $t$ intervals come much closer to attaining
the desired nominal coverage, but they do so at the
expense of being quite long.  The BCa intervals are narrower
than the bootstrap $t$ intervals, but they very often have
coverage levels severely below the nominal target.
This article develops some alternative bootstrap $t$
confidence intervals that retain the favorable asymptotic
coverage while being noticeably narrower than the customary bootstrap $t$
intervals in a collection of small sample settings. These methods
attain much higher coverage than the \bca\ intervals.

One explanation for the wide intervals from the bootstrap $t$
method is that (as described below) some resamplings place zero weight
on some of the $n$ observations.  When one or more extreme
observations gets zero weight, the resampled standard deviation
becomes small which can make the sample mean appear to be an
outlier. We then get very wide tails for the resampled $t$
distribution, including atoms at $\pm\infty$.
The result is wide confidence intervals
for the mean and consequently, high coverage.

We can eliminate the zero weights by merging
the Bayesian bootstrap \citep{rubi:1981} with the bootstrap $t$. 
That provides asymptotically valid confidence intervals, though it lacks second order
accuracy for one-sided confidence intervals. Unfortunately,
the Bayesian bootstrap version of the bootstrap $t$ intervals does
not have very good small sample coverage in the cases examined
here. Perhaps the good coverage properties of the original bootstrap
$t$ are strongly dependent on those small weights.
A next step is to then to replace the unit exponential
weights of the Bayesian bootstrap by weights from a distribution that has
a singular probability density at the origin in order to get near zero
weights that remain positive.  Strictly positive weights keep
the resampled variance positive unless all $n$ observations are equal.
That positive resample variance then keeps the resampled $t$ value finite.
One such weight distribution is Beta($1/2,3/2$).
We show that it has a second order accuracy property as $n\to\infty$, so
its small sample advantages do not conflict with asymptotic performance.

If having some near zero weights
works well, we might expect that having many near zero
weights could work even better.  One way to do that is
to use a very heavy-tailed weight distribution such
that the largest sample value is often large compared to
the sum of the other $n-1$ values.
That is not viable in the case of bootstrap weights.
As noted below the weights must have finite
variance prior to normalization and that precludes
the largest one from dominating the others, at least for large $n$.
We will also see a reason to have a finite fourth moment for the weights.  
A very heavy-tailed weight distribution, the lognormal,
is seen to give very short bootstrap $t$ confidence intervals, often shorter
than even the \bca\ intervals, and with lower coverage.

It is also possible to resample with even more zero weights
than the multinomial distribution provides.  
The Poisson distribution with mean $1$ has been used
for bootstrap weights in \cite{oza:russ:2001}.  While its asymptotic probability
of a zero weight matches that of multinomial sampling,
it has higher probability of a zero weight for the small $n$
that we consider here.
\cite{hart:1969} considers sampling from all $2^n$ subsets of
the data which we can implement by giving each observation
weight $0$ or $1$ with probability $1/2$ each.  To keep
the mean weight equal to one, we can equivalently give
weight $0$ or $2$. \cite{owen:eckl:2012} refer to this as
the double-or-nothing bootstrap.  Here we use it in
a bootstrap $t$ along
with a half-sampling bootstrap $t$ described below.
These methods with more exact zero weights
give even wider confidence intervals
than the multinomial bootstrap $t$ does.  As a result, we can 
get a spectrum of bootstrap $t$ intervals ranging from some
that are even narrower than the \bca\ intervals 
to others that are wider than the usual multinomial bootstrap $t$.

For continuous data, it is easy to see that the probability
of an infinite resampled $t$ statistic is $n^{1-n}$ which
rapidly becomes negligible compared to a threshold
like $0.025$ as $n$ increases.  What is less
remarked on is that this probability can decay much more
slowly for discrete data.  In some computations below we see
infinite length intervals for the usual bootstrap $t$ with
$n=19$ for the mean of $\dpois(1)$ data.  With $\dbeta(1/2,3/2)$ weights,
the bootstrap $t$ statistic is finite. It can have heavy tails
for small $n$. However once there are three or more distinct
observations, or two distinct observations seen at least
three time each, the expected squared bootstrap $t$ value
given the data, becomes finite.

Most of what is known about bootstrap confidence intervals comes
from asymptotic expansions as $n\to\infty$, using assumptions about the
moments of the sampling distribution.   
Even for just one given bootstrap method there can be 
significant differences in small sample
performance for different distributions having the same low order moments.
An extreme version of this arises in the behavior of the bootstrap for
discrete data.  A Cram\'er condition that
\begin{align}\label{eq:cramer}
\limsup_{|t|\to\infty}|\e( \exp(itx))|<1
\end{align}
plays a large role in the asymptotic theory of \cite{hall:1986}.
The assumption provides more tractable Edgeworth expansions.
Equation~\eqref{eq:cramer} rules out cases
where the support of $x$ is confined to an arithmetic sequence,
so it is violated for integer-valued random variables
such as counts.

We are left with an unsatisfactory situation.  The cases with small $n$
have the most consequential differences and the least theoretical understanding.
The study for small values of $n$ has primarily
been done numerically.  Numerical investigation still seems to be
necessary, nearly fifty years after the bootstrap
appeared.  

As mentioned above, strictly positive weights keep the interval lengths finite.
Some other advantages of strictly positive reweighting arise in other contexts.
\cite{xu:etal:2020} mention several.
When every original data point is at least partially represented in every
bootstrap sample, then design matrices for regression that are non-singular
remain nonsingular in every bootstrap sample.  If a binary response is not
linearly separable as a function of the predictors  in the sample, then this 
continues to hold true in bootstrap
samples where every observation has positive weight. 
They also mention similar advantages for survival
analysis problems where the majority of the lifetimes
have been censored. \cite{mase:owen:seil:2024} use positive weights for a variable
importance problem so that while bootstrapping the importance of 
variable $j$ to an outcome for subject $i$, that subject $i$ is represented
in every bootstrap sample. \cite{owen2009monte} finds that quasi-Monte
Carlo sampling with the Bayesian bootstrap weights is very effective
for bootstrap bias estimation.

A great difficulty in comparing nonparametric ACIs for the mean is that there is no universally accepted
figure of merit with which to compare them.  It is impossible to get exactly the nominal
coverage and so we must 
choose between methods with higher coverage 
and methods with shorter intervals.
There are many rules of thumb and default choices in statistics
but none that say how much wider an interval could reasonably be in return
for reducing undercoverage by 1\%.
Without a quantitative answer we can only look at Pareto frontiers and make subjective
judgments.

An outline of this paper is as follows.
Section~\ref{sec:background} sets up the notation, describes \bca\ intervals
and weighted bootstrap $t$ intervals and gives some choices for the weights.
Section~\ref{sec:secondorder} shows that a bootstrap $t$ using Beta($1/2,3/2$)
weights meets the criteria for second order accuracy in \cite{hall:mamm:1994}.
It also has a finite expected mean squared bootstrap $t$ statistic
given data with three or more distinct values, or two distinct values appearing
at least three times each, and it compares interval lengths
for the case with $n=2$ distinct observations.
Section~\ref{sec:commontask} describes in detail the difficulties in defining a common
task framework to judge ACIs, considering calibration, elicitation,
decision theory and a similar tradeoff that arises for precision and recall in
classification.
Section~\ref{sec:testcases} presents a collection of 
test cases for  investigations of bootstrap ACIs
with numerical comparisons for sample sizes $2\le n\le 20$.
It makes a graphical comparison of the
methods most of which fall on a Pareto frontier trading off
interval length and coverage.
Section~\ref{sec:discussion} presents some conclusions.
Appendix~\ref{sec:proofs} contains some proofs.
Appendix~\ref{sec:smallnlit} surveys some of the literature on
the bootstrap for small $n$.
Appendix~\ref{sec:conventions} describes conventions for
handling some corner cases that only have a meaningfully
large probability when $n$ is as small as we consider here.

\section{Background and notation}\label{sec:background}

We suppose throughout that $x_1,\dots,x_n$ are sampled IID  from
a distribution $F$ with mean $\mu$ and variance $\sigma^2\in(0,\infty)$.
A  two-sided confidence interval for $\mu$, at level $\alpha\in(0,1)$,
is a pair of random variables $L$
and $U$, given as functions of $x_1,\dots,x_n$
for which $\Pr( L\le \mu\le U) = \alpha$. We always have $L\le U$.
A customary default choice is $\alpha =0.95$.
While $0.95$ is not a very high confidence level, it is already
a challenge to get close to this coverage in small samples,
even for two-sided coverage.

The two-sided intervals we consider are formed from a pair
of one-sided intervals.
We select $0<\alpha_1<\alpha_2<1$
and seek random $L$ and $U$ with
$\Pr( \mu < L) =\alpha_1$ and $\Pr(\mu\le U)=\alpha_2$
so that $\Pr( L \le \mu \le U) = \alpha_2-\alpha_1$.
For equal one-sided coverage errors we take $\alpha_1=(1-\alpha)/2$
for $\alpha>1/2$ and $\alpha_2 = 1-\alpha_1=(1+\alpha)/2$.
Then the customary values are $\alpha_1=0.025$ and $\alpha_2=0.975$.

Although we want an exact confidence interval for $\mu$, we can
only have one under special circumstances.  Confidence intervals are possible
within parametric families of distributions, but then in addition
to $1-\alpha$ uncertainty from the interval, we have uncertainty
about the extent to which model violations affect our interval.
In a Bayesian setting we may well obtain a credible interval
with our desired coverage level. That calibration extends
to others who share the same prior distribution, but not universally. 
Bayesian methods are often used pragmatically with
noninformative or weakly informative priors 
in order to improve upon frequentist methods.
For example, \cite{gree:pool:2013} note that under weak priors
``95\% confidence intervals approximate 95\% posterior probability intervals''.
Such intervals are also not exact and we face the same tradeoff
of calibration versus length.

While we prefer a nonparametric confidence interval, those are
essentially ruled out by a theorem of \cite{baha:sava:1956}.
Here is the description from \cite{err4qmc}:
\begin{quote}They consider a set $\cf$ of distributions on $\real$.
Letting $\mu(F)$ be $\e(Y)$ when $Y\sim F\in\real$, their conditions are:
\begin{compactenum}[\quad (i)]
\item For all $F\in\cf$, $\mu(F)$ exists and is finite.
\item For all $m\in\real$ there is $F\in\cf$ with $\mu(F)=m$.
\item $\cf$ is convex: if $F,G\in\cf$ and $0<\pi<1$,
  then $\pi F+(1-\pi)G\in \cf$.
\end{compactenum}
Then their Corollary 2 shows that a Borel set constructed
based on $Y_1,\dots,Y_N\stackrel{\mathrm{iid}}
\sim F$
that contains $\mu(F)$ with probability at least $1-\alpha$
also contains any other $m\in\real$ with probability at least $1-\alpha$.
More precisely: we can get a confidence set, but not a useful one.
They allow $N$ to be random so long as $\Pr(N<\infty)=1$.
\end{quote}

In special settings, we know that $0\le x\le 1$ and in those
cases it is possible to construct confidence intervals using
Hoeffding's theorem or using empirical Bernstein intervals.
See \cite{waud:ramd:2024} for a survey and some new
methods that are asymptotically narrower
than predecessors. While these methods attain at least
the desired coverage, they are much wider than the
intervals based on the CLT \citep{aust:mack:2022}.
Similarly, if $x$ has a symmetric distribution with $\e(|x|)<\infty$
then $\e(x)$ equals the median of $x$ and it is possible to
get a confidence interval for it.

When $F$ does not belong to a known parametric family and $x$
does not satisfy some special side condition (such as known bounds 
or symmetry) then we must rely instead on
approximate confidence intervals (ACIs). These are usually based on
the central limit theorem where 
$\sqrt{n}(\bar x-\mu)/\sigma\dotsim\dnorm(0,1)$
for `large enough' $n$. There is no sharp distinction where
we can say that $n$ is large enough while $n-1$ isn't.
A very common rule of thumb is to consider $n\ge30$
large enough.  \cite{kufs:2010} looks into the history of this
rule of thumb. \cite{mill:1986} even suggests that $n\ge10$
may be enough, if one can accept mild skewness in $\bar x$.
On the other hand \cite{hest:2015} has stringent conditions
(one-sided coverages within $0.0025$ of $0.025$)
where Student's $t$ needs $n\ge 4815$ for $x_i\simiid\dexp(1)$.
Based on some simulations,
\cite{boos:hugh:2000} suggest $n\ge (5.66\gamma)^2$
to get nominal 95\% intervals Student's $t$ intervals with two-sided coverage at least $94$\%,
when $x_i$ have skewness $\gamma$. \cite{koha:etal:2014}
find that for some highly skewed electronic commerce data
this guideline requires $n\ge114{,}000$.  These higher numbers
provide good motivation for further investigation of confidence intervals.

The suitability of these rules of thumb depends on how
non-Gaussian the distribution of $x$ is and how harmful it is to have
less than nominal coverage.  In practice we must choose an ACI
without precise information about these.  

We will look at performance of bootstrap ACIs for $n$
in the range from $2$ to $20$. It is not reasonable to expect
good results for $n=2$.  Nor is there a sharp line where
we can say that a reasonable result can be obtained 
for $n\ge n_0$ but not for  $n\le n_0-1$.  
The bootstrap methods we consider will all perform very
well as $n\to\infty$ and will all be problematic for $n=2$
on difficult distributions.
What we can do is observe the trajectory of their length
and coverage as a function of $n$ and see which
of them encounter less difficulty when $n$ is small.
Investigators may have some idea which example settings
are most similar to their problem, even if that knowledge
cannot be expressed precisely.

\subsection{Approximate confidence intervals}

The ACIs we consider are based on
the sample mean and variance
$$
\bar x = \frac1n\sum_{i=1}^nx_i\quad\text{and}\quad
\hat \sigma^2 = \frac1n\sum_{i=1}^n(x_i-\bar x)^2,
$$
and bootstrap versions of these quantities.
The unbiased sample variance estimate is $s^2 = n\hat\sigma^2/(n-1)$.
It figures in the sample $t$ quantity
$$
t = \sqrt{n}\frac{\bar x-\mu}s
$$
in which $\mu$ is a hypothesized value for $\e(x)$.

The best known ACI for the mean has upper and lower
limits
$$
\bar x \pm s t^{1-\alpha/2}_{(n-1)}/\sqrt{n}
$$
where $t^{1-\alpha/2}_{(n-1)}$ is the $1-\alpha/2$ quantile
of Student's $t$ distribution on $n-1$ degrees of freedom.
It satisfies
$$
\lim_{n\to\infty}
\Pr\Bigl(
\bar x  -\frac{s}{\sqrt{n}}t^{1-\alpha/2}_{(n-1)}
\le 
\mu\le\bar x  +\frac{s}{\sqrt{n}}
t^{1-\alpha/2}_{(n-1)}
\Bigr) = 1-\alpha
$$
by the central limit theorem and it is exact when
$x_i\simiid\dnorm(\mu,\sigma^2)$ and $n\ge2$.

For small samples from a distribution $F$ that includes atoms
we must allow for the possibility that $\hat\sigma=0$.
We adopt a convention that the resulting ACI
is simply the point $[\bar x,\bar x]$. Any value of $\mu$ other
than $\bar x$ would have given $|t|=\infty$.
This is the first of several conventions that we need to
use in order to handle very small $n$.
They are collected in Appendix~\ref{sec:conventions}.

\subsection{Bootstrap $t$}
Here we present the bootstrap $t$ algorithm of \cite{efro:1981}.
It is also known as the percentile $t$ algorithm
\citep{hall:1986}.

For our purposes of generalizing the bootstrap
$t$, it is more convenient to use $\hat\sigma^2$
than $s^2$.
The corresponding $t$ statistic is
$$
\tilde t = \sqrt{n}\frac{\bar x-\mu}{\hat\sigma}
= \sqrt{\frac{n}{n-1}}t.
$$
To get a bootstrap sample $x_1^*,\dots,x_n^*$ we draw indices
$j(1),\dots,j(n)$ IID from $\dunif\{1,2,\dots,n\}$ and set
$x_i^* =x_{j(i)}$ for $i=1,\dots,n$.  
This amounts to resampling the original $x_i$
values with replacement. From the bootstrap data we can compute
\begin{align}\label{eq:boott}
\bar x^* = \frac1n\sum_{i=1}^nx_i^*,\quad
\hat\sigma^{*2} = \frac1n\sum_{i=1}^n(x_i^*-\bar x^*)^2\quad
\text{and}\quad
t^* = \sqrt{n}\frac{\bar x^*-\bar x}{\hat\sigma^*}.
\end{align}
The bootstrap $t$ ACI is based on equating the
unknown distribution of $\tilde t=\sqrt{n}(\bar x-\mu)/\hat\sigma$
under $x_i\simiid F$ to the distribution of $t^*$ from
equation~\eqref{eq:boott}. 
We write $t^*$ instead of $\tilde t^*$ since the presence
of the asterisk already indicates that it is a bootstrap
quantity.
The distribution of $t^*$ can be computed to any desired level of accuracy
by sampling data sets $(x_1^*,\dots,x_n^*)$ some number $B$ of times.  A common
choice is $B=2000$.   The resulting ACI is
\begin{align}\label{eq:bootci}
\bigl[\bar x - \hat\sigma t^{(*\alpha_2)}/\sqrt{n}
,\bar x - \hat\sigma t^{(*\alpha_1)}/\sqrt{n}\,\bigr]
\end{align}
where $t^{(*\alpha_j)}$ is the $\alpha_j$ quantile of the distribution
of $t^*$.  It might be more precise to write 
$\hat t^{(*\alpha_j)}$, as it is an estimated quantile
but we will write $t^{(*\alpha_j)}$ to get 
less cluttered notation.

\cite{hall:mamm:1994} consider several ways to define a bootstrap $t$.
When their weights $N_{ni}^*$ are normalized to sum to $n$,
like $nw_i$ here, then their $\hat\sigma^{*2}$ on page 2017
reduces to $\hat\sigma^{*2}$ above.  They refer
$(\bar x-\mu)/\hat\sigma$ to $(\bar x^*-\bar x)/\hat\sigma^*$ 
which is the same bootstrap approximation we use here.
\cite{maso:newt:1992} study a bootstrap $t$ method where
$\hat\sigma^*$ is multiplied by a standard deviation of the
resampling weights.  The multinomial bootstrap $t$ is not
then a special case of their quantity.
\cite{csor:mart:nasa:2024} consider this same statistic
as well as $\sqrt{\sum_{i=1}^nv_i}(\bar x^*-\bar x)/\hat\sigma^*$.
That latter statistic does generalize $t^*$.  They show consistency
but do not consider convergence rates.

The bootstrap $t$ is based on a plug-in principle.
If we knew the distribution of $(\bar x-\mu)/\hat\sigma$
for $x_i\simiid F$,
we could use it along with the known values of $\bar x$ 
and $\hat\sigma$  to get exact confidence intervals for $\mu$.
We don't know that distribution but we can get
the distribution of $(\bar x^*-\bar x)/\hat\sigma^*$
for $x_i^*\simiid \wh F$ where $\wh F$ is the empirical
distribution of $x_1,\dots,x_n$.  So we `plug-in' $\wh F$ for $F$
to approximate the distribution of $(\bar x-\mu)/\hat\sigma$
by that of $(\bar x^*-\bar x)/\hat\sigma^*$ and this gives
us an ACI for $\mu$.

There is positive probability that $\hat\sigma^*=0$.
This will ordinarily give $t^* =\pm\infty$, resulting in wider ACIs.
With no ties in the data, $\Pr( \hat\sigma^* =0)=n^{1-n}$ 
which for moderately large $n$ is typically much
smaller than $\min(\alpha_1,1-\alpha_2)$. Then retaining
some infinite values of $|t^*|$ still results in a finite
ACI length that is wider than it would be if we had omitted the infinite cases.
Some weighted bootstraps that we consider
below have larger values of $\Pr(\hat\sigma^*=0)$ than the
usual bootstrap $t$ has, and we see wider ACIs for them.
For discrete data, such as counts, the possibility of $\hat\sigma^*=0$
is greatly increased.

When there are ties in the data, it is possible to have $\bar x^* - \bar x = \hat\sigma^*=0$.
We adopt a convention that $t^*=0$ in this case.
For a method  with more frequent occurences of $\bar x^* - \bar x = \hat\sigma^*=0$,
the distribution of $t^*$ will have a larger atom at $0$.
When $\Pr( x=\mu)=0$, this will result in
shorter ACIs with reduced coverage of the true mean, leading
us to prefer a different method.

Instead of the bootstrap $t$, one could simply use the quantiles of
$\bar x^{(*\alpha_j)}$ themselves to form an ACI for $\mu$.
This is known as the percentile method. 
It is not competitive with the bootstrap $t$ method when we want an ACI for the mean.
It is not as well regarded as the \bca\ method described next.
The percentile method and some other bootstraps are described in \cite{efro:tibs:1993}.

\subsection{Bias corrected accelerated bootstrap}\label{sec:bca}

The bias corrected accelerated, or \bca, ACI
is from~\cite{efro:1987}.  Here we consider the version used
in the R function {\tt bcanon} from the {\tt bootstrap}
package \cite{R:bootstrap} based on S functions by R.\ Tibshirani
for the book \cite{efro:tibs:1993}. We describe it for a sample mean though
it can be applied to more general functions.

The \bca\ interval uses the same values $x_1^*,\dots,x_n^*$ and their
average $\bar x^*$ that the bootstrap $t$ does.
Given $B$ values of $\bar x^*$ it sets a bias adjustment parameter
\begin{align}\label{eq:defz0}
z_0 = \Phi^{-1}\Bigl( \frac1B\sum_{b=1}^B 1\{ \bar x^{*b} <\bar x\}\Bigr).
\end{align}
Then it  uses leave one out means
$\bar x_{-i} = (n\bar x-x_i)/(n-1)$ 
to define an acceleration quantity
$$
a = \frac{\sum_{i=1}^nu_i^3}
{6\bigl(\sum_{i=1}^n u_i^2\bigr)^{3/2}}
$$
for $u_i =\bar x-\bar x_{-i} = (x_i-\bar x)/(n-1)$.
Then modifying the limits $\alpha_j$  for $j=1,2$ to
$$\tilde \alpha_j = \Phi\Bigl(
z_0  + \frac{z_0+z_j}{1-a(z_0+z_j)}\Bigr)
$$
the lower and upper confidence limits are the quantiles
$\bar x^{*(\tilde \alpha_j)}$ for $j=1,2$.

The \bca\ method is harder to interpret than the bootstrap $t$.
It is derived using the idea that there exists a monotone
increasing transformation $\phi(\cdot)$ for which $\phi(\bar x)\sim \dnorm(\phi(\mu)-z_0\sigma_\phi,
\sigma^2_\phi)$ where $\sigma_\phi = 1+a\phi(\mu)$.
The form of the acceleration is based on an asymptotic argument
to adjust the confidence level to be second order accurate.
\cite{dici:efro:1996} remark that the \bca\ intervals
strike some people as being `unbootstraplike' because the
formula for $a$ does not come from the bootstrap replications. 

To use the  \bca\ interval we only need to be able to calculate our scalar statistic
of interest on bootstrap samples of the data.  
The bootstrap $t$ requires a quantity to use for $\hat\sigma$.
When the statistic of interest is not a sample mean, then the \bca\ intervals
remain easy to use while the bootstrap $t$ can become difficult to use.

\subsection{Bootstrap as reweighting}

In the bootstrap sample $x_1^*,\dots,x_n^*$ the observation
$x_i$ is included a random number $v_i$ of times.
Then $\bar x^* = \sum_{i=1}^n w_i x_i$ where
the weights $w_i$ are normalized 
counts:  $w_i = v_i/\sum_{i'=1}^nv_{i'} = v_i/n$.
The distribution of each $v_i$ is $\dbin(n,1/n)$
and the joint distribution of $(v_1,\dots,v_n)$
is a multinomial.

The Bayesian bootstrap of \cite{rubi:1981} uses different weights.
It samples $v_i\simiid\dexp(1)$ and then takes
$w_i = v_i/\sum_{i'=1}^nv_{i'}$.
The pseudo-counts $v_i$ are positive non-integer values.
This bootstrap  has a connection to Bayesian statistics
and is usually used in percentile methods.

We can write the bootstrap $t$ method in terms of
weighted resampling for general  bootstrap weights. For that, we take
$$
\hat\mu^* = \sum_{i=1}^nw_ix_i,\quad
\hat\sigma^{2*} = \sum_{i=1}^nw_i(x_i-\hat\mu^*)^2
\quad\text{and}\quad
t^{*}
 = \sqrt{n}\frac{\hat\mu^*-\bar x}{\hat\sigma^*}.
$$
Then the ACI is
\begin{align}\label{eq:btaci}
\Bigl[\bar x-\frac{\hat\sigma}{\sqrt{n}}t^{(*\alpha_2)},
\bar x-\frac{\hat\sigma}{\sqrt{n}}t^{(*\alpha_1)}
\Bigr].
\end{align}
The quantiles of $t^*$ are estimated by sampling pseudo-counts,
normalizing them into weights and then computing the weighted
mean, variance and $t$ statistic.
The bootstrap $t$ ACIs we consider take the form~\eqref{eq:btaci} for
different distributions of $(v_1,\dots,v_n)$.

In the weighting approach we are not exactly mimicking the sampling process that lead to
the original data $x_i$ by plugging in $\wh F$ for $F$.  We are however retaining
the idea of randomly reweighting the data that the original bootstrap has.
In that sense, it is bootstraplike.

We can use some other weights than the multinomial or exponential ones
but the weights cannot be completely arbitrary.
For properly calibrated asymptotic coverage it suffices
for IID pseudocounts $v_i$ to have mean $1$ and
variance $1$.  For second order accuracy we also want
skewness one \citep{hall:mamm:1994}.
That is, we want $\e(v_i)=\e((v_i-1)^2)=\e((v_i-1)^3)=1$.
We will only consider $v_i\ge0$, because then $\Pr(\hat\sigma^{*2}<0)=0$.

The assumption that $\e(v_i)=\var(v_i)=1$ has a consequence in terms of
effective sample size.  Conditionally on the values of pseudo-counts $v_i$, the variance of 
the weighted mean $\sum_{i=1}^nv_ix_i/\sum_{i=1}^nv_i$ is $\sigma^2\sum_{i=1}^nv_i^2/(\sum_{i=1}^nv_i)^2$.
Equating that to $\sigma^2/n_{\eff}$ gives an effective sample size of
$$
n_{\eff} = \frac{(\sum_{i=1}^nv_i)^2}{\sum_{i=1}^nv_i^2}.
$$
We find that $n_{\eff}/n\to 1/2$ as $n\to\infty$ when $v_i$ are IID with mean $1$ and variance~$1$.
Thus we may think of the bootstrap as a reweighting the data using an effective sample
size of about $n/2$.

The Poisson distribution with mean $1$ has variance and skewness also equal to $1$.
\cite{oza:russ:2001} 
use $v_i\simiid\dpois(1)$ in online bootstrap aggregation for machine learning.
For large $n$, the multinomial weights of the bootstrap are almost equivalent to using
independent Poisson weights with mean $1$. Independent Poisson weights are more convenient
than multinomial weights in a streaming environment.

The random variables $v_i\simiid\dunif\{0,2\}$ have mean $1$ and variance $1$
and they can be used in a bootstrap.
For even $n$, we can similarly select $n/2$ of the $v_i$ to be $2$ in a simple random
sample and let the other $v_i$ be $0$. Neither of these will lead to a second
order accurate bootstrap.
These methods can be derived using the typical value theory of \cite{hart:1969}.
A set of $N-1$ distinct `typical' values partitions $\real$ into $N$
intervals (two of which have infinite length) such that each of the $N$ intervals
has probability $1/N$ of containing the parameter of interest.  For a symmetric
distribution we can take means of  $N=2^n-1$ different  subsamples of the
data and then form a confidence interval from the union of a `central'
subset of those intervals.

For large $n$, normalized $\dunif\{0,2\}$ weights correspond very closely to the
half-sampling method used in survey sampling \citep{mcca:1969}
which takes simple random samples of half of the data points.
In our bootstrap $t$ version of half-sampling, when $n$ is even
we select $n/2$ of $n$ observations by a simple random sample
where all ${n \choose n/2}$ selections being equally probable.
When $n$ is odd we take a simple random sample of either $(n+1)/2$ points
or $(n-1)/2$ points, each with probability $1/2$.

The Poisson and double-or-nothing weights can both give
$v_1=v_2=\cdots=v_n=0$ in which case the weights $w_i$
are not well defined. For double-or-nothing this happens
with probability $2^{-n}$ while for the Poisson weights
it happens with probability $\exp(-n)$.  The chance of this
happening is miniscule for modestly large $n$ but for smaller
$n$ this probability is not negligible compared to
customary values of  $\alpha_1$ and $1-\alpha_2$.
A reasonable convention for this case is to discard such weights
and then sample $(v_1,\dots,v_n)$ until $B$ cases with  $\max_i v_i>0$
have been obtained.

In his work on the wild bootstrap, \cite{mamm:1993} presents a distribution on just two support
points that has mean zero, variance one and skewness one.  Adding one to this random variable
gives a  positive random pseudo-count  $v$ 
that satisfies the conditions for second order accuracy. Specifically, if
\begin{align}\label{eq:mammenpseudocounts}
\Pr\Bigl(v = \frac{3\pm\sqrt{5}}2\Bigr)  = 
\frac12\mp\frac1{2\sqrt{5}},
\end{align}
then $\e(v)=\var(v)=\e((v-1)^3)=1$.

We can combine 
$\e(v)=\var(v)=\e((v-1)^3)=1$ with $\Pr(v=0)>0$
by taking $v=0,1,3$ with probabilities $1/3,1/2,1/6$
respectively.  This $v$ has kurtosis $0$.  Another solution takes $v=0,1,2,4$
with probabilities $9,8,6,1$ divided by $24$.  It has kurtosis $1$
like the Poisson distribution.
These combine second order accuracy with an
atom at zero. We don't consider them further because
Poisson weights already have both of those properties.

There has not been much attention paid to the kurtosis of the
pseudo-counts.  One exception is \cite{guil:1999}
who gives a kurtosis condition that makes the
one-sided coverage errors $O(n^{-3/2})$ for studentized
statistics like we study here. In other words, she can obtain
third order accuracy.  In addition to a Cramer
condition, Guillou requires the sample skewness of $v_i$
to be $1+O_p(n^{-1/2})$ and also requires the sample kurtosis
of the $v_i$ to satisfy a relation with the sample kurtosis of the~$x_i$.

\subsection{Weights that are singular at the origin}\label{sec:betabootstrap-t}

We would like a distribution for pseudocounts $v$ with $\Pr(v>0)=1$
but for which $v$ takes values very close to zero with high
probability, by having a probability density function that diverges
to infinity as $v\to0$.  The $\dbeta(a,b)$ distributions with first parameter $0<a<1$
are singular at the origin, and scaling them to have mean and variance
equal to one, generates a parametric family of distributions indexed by $a$.

\begin{proposition}
For $0<a<1$, and $b=(a^2+a)/(1-a)$ let
$x\sim\dbeta(a,b)$ and $v=Ax$ for $A=2/(1-a)$.
Then $\e(v)=\var(v)=1$ and $v$ has skewness $2a$.
\end{proposition}
\begin{proof}
First  $a+b=2a/(1-a)=aA$ so
$$\e(x) = \frac{a}{a+b} =\frac1A$$
and then $\e(v)=1$.
Next $a+b+1=aA+1=(1-2/A)A+1=A-1$, and so
$$
\var(x) = \frac1A\Bigl(1-\frac1A\Bigr)(a+b+1)^{-1}
 = \frac1A\Bigl(1-\frac1A\Bigr)\frac1{A-1}=\frac1{A^2}
$$
so $\var(v)=1$.
Finally, $v$ and $x$ both have skewness
\begin{align*}
\gamma(a)=\frac{2(b-a)\sqrt{a+b+1}}{(a+b+2)\sqrt{ab}}
\end{align*}
Now $a+b+1=(1+a)/(1-a)$, $b-a=2a/(1-a)$ and $ab=a^2(a+1)/(1-a)=a^2(a+b+1)$, so
$$
\gamma(a) = \frac{2(b-a)}{(a+b+2)a}
= \frac{4a(1-a)^{-1}}{2(1-a)^{-1}a} =2a
$$
as required.
\end{proof}

There are a few convenient special cases.  First, taking $a=1/2$
we get $v = 4\times\dbeta(1/2,3/2)$ with mean, variance and
skewness all equal to $1$.  
Theorem~\ref{thm:2ndorder} below shows that this distribution
satisfies the conditions of \cite{hall:mamm:1994} for second
order accuracy. We call the resulting method the beta bootstrap $t$.

Another convenient choice arises with $b=1$.  We get this by taking
$a=\sqrt{2}-1\doteq 0.41$ and $A=2+\sqrt{2}$.  Then the probability density
of $v$ is 
\begin{align}\label{eq:defpowerdistn}
f(v) = \frac{a}{A^{a}}v^{a-1}1_{0<v<A}.
\end{align}
We call the resulting method the power bootstrap $t$
as $f$ is simply a negative power of $v$, scaled and truncated to $(0,A)$.
Properties of these beta distributed weights and some other weights are summarized in Table~\ref{tab:pcounts}.

Smaller values of $a$ correspond to beta distributions with
even stronger singularities at the origin. We see below that
generally results in  longer ACIs and higher coverage.
For instance the power bootstrap $t$ typically has longer
intervals than the beta bootstrap $t$.
If we take the limit as $a\to0$ then $A\times\dbeta(a,b)$ converges
in distribution to $\dunif\{0,2\}$, so the double-or-nothing
bootstrap arises as the limit of this beta family.
Some other interesting examples are $3\times\dbeta(1/3,2/3)$
and $6\times\dbeta(2/3,10/3)$
because they have simple expressions with round numbers.
When  $a<\sqrt{2}-1$, we get $b<1$ and then the density is
no longer monotonic decreasing in~$v$.
For $a\ge1$, it is impossible to scale the beta distribution
to have mean and variance both equal to one as this would
give $b\not\in(0,\infty)$.
In a mixture model with random $a\in(0,1)$ the variable
$v\sim A(a)\dbeta(a,b(a))$, has mean variance and
skewness equal to $1$ if $\e(a)=1/2$.

Some gamma distributions have unbounded probability density
functions at the origin.  However the only gamma distribution
with mean one and variance one is the exponential distribution
used already in the Bayesian bootstrap.
It has a bounded density and also has skewness $2$.

\begin{table}
\centering
\begin{tabular}{llccccc}
\toprule
Method & Distribution & IID &$\e(v)$ & $\var(v)$ & $\skew(v)$ & $\Pr(v=0)$\\[1ex]
\midrule
Original & Multinomial & & 1 & $1-1/n$ & $\frac{1-2/n}{\sqrt{1-1/n}}$ & $(1-1/n)^n$\\[1ex]
Half sampling & Simple random& &1 & 1 & 0 & $1/2$\\[1ex]
Bayesian & $\dexp(1)$ & $\checkmark$ & 1 & 1 & 2 &0\\[1ex]
Poisson & $\dpois(1)$ & $\checkmark$ & 1 & 1 & 1 & $e^{-1}$\\[1ex]
Beta &$4\times\dbeta(1/2,3/2)$&$\checkmark$& 1 & 1 & 1 & 0\\[1ex]
Power & Equation~\eqref{eq:defpowerdistn}
& $\checkmark$ & 1 & 1 & $2(\sqrt{2}-1)$ & 0\\[1ex]
Mammen & Equation~\eqref{eq:mammenpseudocounts} 
&$\checkmark$& 1 & 1&1&0\\[1ex]
Double/nothing &$\dunif\{0,2\}$& $\checkmark$ &1 & 1 & 0 & $1/2$\\[1ex]
Lognormal & $\exp(\dnorm(-\log(2)/2,\log(2)))$ & $\checkmark$ &1&1 &4 &0\\
\bottomrule
\end{tabular}
\caption{\label{tab:pcounts}
Properties of the individual (pseudo-)counts $v_i$ in different bootstrap
weighing schemes. 
}
\end{table}

\subsection{Heavy-tailed weights}

A useful example of a heavy-tailed distribution is the log-normal distribution.
By taking $v_i = \exp( -\log(2)/2 + \sqrt{\log(2)}Z)$ for $Z\sim\dnorm(0,1)$
we get $\e(v_i)=\var(v_i)=1$.  Bootstrap $t$ ACIs with $v_i$ from this log-normal
distribution will not be second order accurate because the skewness of $v_i$
is not equal to one.  There does not seem to be a common distribution
of positive $v$ with very heavy tails and mean, variance and skewness
all equal to one.  A bootstrap $t$ with such a distribution would be second order accurate.
We will see that log-normal weights lead to very short ACIs
and low coverage.  

It is reasonable to expect low coverage for other
heavy-tailed distributions too.  
One other heavy-tailed sampler takes $\Pr(v=\sqrt{n})=1/\sqrt{n}$
with $\Pr( v=1)=1-1/\sqrt{n}$, for use in~\eqref{eq:btaci}.
This sampler is modeled on the `butcher-knife' that \cite{hest:1995} 
describes as an alternative to the jackknife.
In some limited investigations it yielded very short intervals
and it is not included in the numerical results of this paper.

\section{Properties of the beta bootstrap~$t$}\label{sec:secondorder}

Here we show that the beta bootstrap $t$ is second
order accurate as $n\to\infty$. We also 
give conditions for $\e( (t^*)^2\giv x_1,\dots,x_n)<\infty$
that apply for small $n$.
Finally, we compare the widths of $95\%$ bootstrap confidence
intervals for the case with $n=2$ and $x_1\ne x_2$,

\subsection{Second order accuracy}
We show that the bootstrap $t$ with pseudo-counts
$v_i\simiid 4\dbeta(1/2,3/2)$ satisfies a critical condition
from \cite{hall:mamm:1994} for second order accuracy.
They assume that $\e(x^{12})<\infty$ and do not impose
the Cram\'er condition~\eqref{eq:cramer}. Instead they
show the moment conditions that must be satisfied.
A Cram\'er condition is an additional requirement.

Their weights corresponding to $v_i$ are $W_i = nv_i/\sum_{i'=1}^nv_{i'}$.
That is, $W_i = nw_i$.
For second order accuracy, \cite{hall:mamm:1994} state
that it is necessary and sufficient to have both
\begin{align}\label{eq:hallmammenconds}
\e( (W_i-1)^2) = 1 + o(n^{-1/2}) 
\quad\text{and}\quad 
\e( (W_i-1)^3) = 1 + o(1)
\end{align}
along with $\e(W_i)=1$ that they earlier assumed.
We easily see that $\e(W_i) = \e(\bar W)$ by symmetry
and $\bar W=1$ where $\bar W=(1/n)\sum_{i=1}^nW_i$.
Therefore equation~\eqref{eq:hallmammenconds} gives the conditions
on $v_i$ that remain to be established.

We consider IID pseudocounts $v_i>0$  with
$\e(v)=1$, $\e(v^2)=2$ and $\e(v^3)=5$.
Then $\e((v-1)^2)=\e((v-1)^3)=1$.
Because $W_i$ are formed as ratios we have to account
for their random denominator when verifying~\eqref{eq:hallmammenconds}.
For $\ell\ge1$, we let $\mu_\ell = \e(v^\ell)$. 

\begin{theorem}\label{thm:2ndorder}
Let $v_i$ be positive IID random variables 
with $\mu_1=1$, $\mu_2=2$
and $\e(v^{-\alpha})\le M<\infty$ for some $\alpha>0$.
If $\mu_{4+\epsilon}<\infty$ for some $\epsilon>0$, then
\begin{align}
\e(W_i^2) &= 2 +\frac{10-2\mu_3}n + O(n^{-2}).\label{eq:mu2}
\end{align}
If $\mu_{6+\epsilon}<\infty$ holds for some $\epsilon>0$ then
\begin{align}
\e(W_i^3) &= \mu_3 +\frac{9\mu_3-3\mu_4}n + O(n^{-3/2}).\label{eq:mu3}
\end{align}
\end{theorem}
\begin{proof}
See Appendix~\ref{proof:thm:2ndorder}.
\end{proof}

The beta bootstrap $t$ of Section~\ref{sec:betabootstrap-t} satisfies
the conditions of Theorem~\ref{thm:2ndorder}.
Because it has $\mu_3=5$, it then satisfies the conditions of
\cite{hall:mamm:1994}.

\subsection{Finiteness of some RMSLs}\label{sec:nis2}

The beta bootstrap $t$ and the power bootstrap $t$
never give infinite length confidence intervals.  
The methods with $\Pr(w_i=0)>0$ give infinite length
intervals for small $n$ even with continuously distributed
data and they can have infinite length for substantially larger
$n$ when the distribution of $x_i$ has atoms.

The finite lengths for the beta and power bootstrap $t$
intervals can still be large.
Here we give conditions to ensure that $\e( (t^*)^2\giv x_1,\dots,x_n)<\infty$.
That is a stronger condition than having a finite conditional quantile.

First we suppose that there are at least three distinct $x_i$
in the sample. This holds with probability one when $x_i$
have a continuous distribution. It also holds with high probability
for some discrete distributions such as the $\dpois(1)$ distribution
where all the bootstrap $t$ methods with $\Pr( v=0)>0$
gave an infinite RMSL for $n=19$ in Section~\ref{sec:testcases}.

\begin{theorem}\label{thm:threedistinct}
Let $x_1,\dots,x_n$ have $k\ge3$ distinct observations,
and let $t^*$ be the beta bootstrap $t$ statistic.
Then $\e( (t^*)^2\giv x_1,\dots,x_n)<\infty$.
The same holds for the power bootstrap $t$ statistic.
\end{theorem}
\begin{proof}
See Appendix~\ref{sec:proof:thm:threedistinct}
\end{proof}

More generally if $v_i$ are IID $\dbeta(a, (a^2+a)/(1-a))$,
for $0<a<1$ and there are more than $1/a$ distinct observations $x_i$
in the sample, then $\e( (t^*)^2\giv x_1,\dots,x_n)<\infty$
holds by the argument behind Theorem~\ref{thm:threedistinct}.

\begin{theorem}\label{thm:pairofthrees}
Let there be two distinct observations among $x_1,\dots,x_n$,
each appearing at least three times.  Then for the beta bootstrap $t$,
$\e( (t^*)^2\giv x_1,\dots,x_n)<\infty$.
The same holds for the power bootstrap $t$ statistic.
\end{theorem}
\begin{proof}
See Appendix~\ref{sec:proof:thm:pairofthrees}.
\end{proof}

Similarly to the prior theorem, it suffices to have
more than $1/a$ copies of each of the two distinct
observations in order to get
$\e( (t^*)^2\giv x_1,\dots,x_n)<\infty$.

\subsection{Width when $n=2$}
For $n=2$, with $x_1\ne x_2$, the proof of Theorem~\ref{thm:pairofthrees}
leads to the expression
\begin{align}\label{eq:tstar}
t^* = \sqrt{2}\frac{w_1-1/2}{\sqrt{w_1(1-w_1)}}\sign(x_1-x_2).
\end{align}
Because $w_1=v_1/(v_1+v_2)$,
this allows us to compute the quantiles of $t^*$ and compare 
the different bootstrap $t$
methods to each other, to Student's $t$ and to \bca.
We can express all of the interval lengths as multiples of $|x_1-x_2|$.

For Student's $t$ with  $n=2$ we have $s^2 = (x_1-x_2)^2/2$.
The usual Student's $t$ interval reduces to
$$
\bar x \pm t^{0.975}_{(1)}\frac{s}{\sqrt{2}}=
\bar x \pm t^{0.975}_{(1)}\frac{|x_1-x_2|}2.
$$
The $t_{(1)}$ distribution is the Cauchy distribution, and we
get $t^{0.975}_{(1)}\doteq 12.71$.  For the very benign case where
$x\simiid \dnorm(\mu,\sigma^2)$, the exact confidence interval
has width $12.71|x_1-x_2|$ when $n=2$.

For $n=2$, the weighted bootstrap $t$ interval widths
are $\hat\sigma = |x_1-x_2|/2$ times
the $0.95$ quantile of the distribution of
$2|w_1-1/2|/\sqrt{w_1(1-w_1)}$.  That is, the width
is the $0.95$ quantile of
$|w_1-1/2|/\sqrt{w_1(1-w_1)}$.  
For the usual multinomial bootstrap $t$, $w_1$ takes values
$0$, $1/2$ and $1$ with probabilities $1/4$, $1/2$
and $1/4$ respectively. Then the width
takes values $0$ and $\infty$ with probability
$1/2$ each and the $95$'th percentile of this distribution is $\infty$.
The Poisson bootstrap $t$, double-or-nothing bootstrap $t$
and half-sampling bootstrap $t$ all give infinite intervals for $n=2$ with $x_1\ne x_2$.

For the \bca\ with $n=2$, the formula in Section~\ref{sec:bca} 
yields $[x_{(1)},x_{(1)}]$ where $x_{(1)}= \min(x_1,x_2)$.
While $\Pr( \bar x^*=\bar x)$ is negligible for modestly
large $n$, it equals $1/2$ for $n=2$.
If we modify equation~\eqref{eq:defz0} to put
half weight on cases where $\bar x^*=\bar x$, then
we get a 95\% interval of $[x_{(1)},x_{(2)}]$
which is a truer reflection of how \bca\ operates.
The numerical results in this paper use that modification.
When $\bar x^*=\bar x$ arises because
the same values were averaged in both cases,
they will also be equal in their floating point representations.

Table~\ref{tab:bootfactor} shows width factors for various
ACIs when $n=2$.  The bootstrap ones were computed numerically
by sampling values of $w_1=v_1/(v_1+v_2)$.  
What we see there is that most of the bootstrap ACIs
we consider are either narrower than Student's $t$ intervals 
or have infinite length when $n=2$.  It is reasonable for
a nonparametric interval to be wider than one based on a Gaussian
data assumption.  Of the methods in Table~\ref{tab:bootfactor},
only the power bootstrap $t$
has intervals wider than Student's $t$ but of finite length.

\begin{table}[t]
  \centering
  \begin{tabular}{lc}
\toprule
Method & Width$/|x_1-x_2|$\\
\midrule
Multinomial & $\infty$\\
Power & 18.60\\
Student & 12.71\\
Beta   & 10.78 \\
Bayesian &\phz3.04\\
Log normal& \phz1.43\\
Mammen & \phz1.12\\
\bca & \,1\phantom{11}\\
\bottomrule
  \end{tabular}
  \caption{\label{tab:bootfactor}
    This table shows widths of $t$-based confidence intervals
    at nominal level 95\%,
    when $n=2$ as a multiple of $|x_1-x_2|$.
    Multinomial is the usual bootstrap $t$.  Student is the
    usual Student's $t$ interval.  The others are weighted
    bootstraps.
    }
\end{table}


\section{Towards a common task framework for ACIs}\label{sec:commontask}


The considerable progress in black box prediction and natural language
processing has been attributed to the common task framework
epitomized by DARPA's challenge problems.
That framework includes standard benchmark problems on which to
compare objectively measurable performance such as predictive accuracy.  
In those settings, state of the art performance need not always be
accompanied by a full explanation of why a method works well.
See \cite{dono:2024} who points to a lecture by Mark Liberman for the term.

There is no widely accepted measure
of performance for ACIs, so there is no accepted benchmark
with which to judge state of the art performance.
The main difficulty is that we have two competing objectives: attaining something
close to the nominal coverage level, and having intervals that are as short as possible.
Sometimes one method has better coverage than another along with shorter
intervals. Then that other method is inadmissible.  Such domination is
rare among seriously considered methods, and they span a large Pareto
frontier.

Confidence intervals are similar to prediction intervals and also
to quantile elicitation, where there are established criteria.
Binary classification has an established $F_1$ score to make a similar tradeoff
between precision and recall.
There are also decision theoretic methods
to measure performance of confidence intervals.  
We see below how none of these concepts provide a good benchmark for ACIs.

\subsection{Calibration}
For both confidence and prediction intervals,
we provide an ideally short interval $[L,U]$ and there is a desired
nominal coverage probability $\alpha$. 
Here $L=L(X)$ and $U=U(X)$ where $X$ represents all the data.
In a prediction interval we seek $\Pr( L\le Y\le U\giv X) =\alpha$
for random $Y$ conditionally on $X$.  In a confidence interval
we seek $\Pr( L(X) \le \theta\le U(X))$ for a non-random quantity
$\theta$ that is a function of the distribution of $X$, such as
the mean $\mu$. The probability is under randomess in $X$.

Prediction intervals can be calibrated on an exchangeable set
of problem instances in order to get close to the nominal
coverage. For a recent survey of calibration
based on conformal inference, see \cite{ange:bate:2023}.
There, in an exchangeable sequence of $N$ prediction problems
it is possible to calibrate so that the coverage of the $n$'th
one is between $1-\alpha$ and $1-\alpha+1/(n+1)$.
We can compare well-calibrated prediction intervals by preferring shorter
ones using mean length or some other aspect of 
the distribution of their interval lengths.

Confidence intervals are harder to calibrate than prediction
intervals.  We don't generally observe the true parameter value.  
We can calibrate them numerically for specific hypothesized
distributions of $x_i$.
We usually lack a way to show that our problem is 
exchangeable within some set of hypothesized problems
that we can calibrate.

\subsection{Elicitation and the interval score}

Continuing the comparison between confidence intervals and prediction
intervals, the one-sided version of a prediction interval is a quantile
$Q^\alpha$ satisfying $\Pr( x\le Q^\alpha)=\alpha$, which for simplicity
we will assume is unique.
Quantiles minimize the expected value of a piecewise linear `pinball' score function,
familiar from quantile regression \citep{koen:2005}, given by
\begin{align}\label{eq:quantilescore}
S_\alpha(\theta;x) = (x-\theta)\bigl( \alpha- 1_{x<\theta}\bigr).
\end{align}
The minimizer over $\theta$ of $\e(S_\alpha(\theta;x))$ is the
$\alpha$ quantile of $x$ very generally.  
\cite{gnei:raft:2007} describe a requirement that either
$\e( \max(x,0))<\infty$ or $\e( \max(-x,0))<\infty$,
both of which are satisfied if $\e(|x|)<\infty$.

The literature on scoring mixes settings where the score function is
a utility to be maximized with those where it is a loss to be minimized.  We will use
losses, translating some of the results for utilities to this case.

A scoring function $S$ is consistent for a real-valued functional $T$ of the distribution $F$
relative to the class $\cf$ of distributions if
$\e_F(S(\theta,x))\ge\e_F(S(T(F),x))$ for all $F\in\cf$ and all  $\theta\in\real$ when $x\sim F$.
It is strictly consistent if equality holds only when $\theta=T(F)$.
If $T(\cdot)$ has a strictly consistent scoring function, then it is elicitable
(relative to the class $\cf$). For the distinction between consistent
score functions and the closely related proper score functions, see \cite{gnei:2011}.

Quantiles are elicitable because
$S_\alpha$ above is a consistent score function.
This function can be used
to elicit an expert's opinion of the quantile of $x$.
When $\e(\cdot)$ is with respect to that expert's opinion
of the distribution of $x$, then their best choice is to
answer with their true opinion about $Q^\alpha$
instead of, for example, cautiously over-estimating the probability
of some undesired outcome (such as rain or financial loss or illness). 
Any such attempts to game the problem can only worsen the expert's score.

We can form a prediction interval from two quantiles such
as $Q^{\alpha_j}$ for $j=1,2$.
For a two-sided symmetric prediction interval at
level $\alpha$, we can take $\alpha_1=(1-\alpha)/2$
and $\alpha_2=1-\alpha_1$ and sum the scores getting 
\begin{align*}
S_{\alpha_1}(x,\theta_1)+S_{\alpha_2}(x,\theta_2)
&=(\alpha_1-1_{x\le\theta_1})(x-\theta_1)
+(1-\alpha_1-1_{x\le\theta_1})(x-\theta_2)\\
&= \alpha_1(\theta_2-\theta_1)+1_{x\le\theta_1}(\theta_1-x)+1_{x>\theta_2}(x-\theta_2).
\end{align*}
The interval length between the upper and lower quantiles is $\theta_2-\theta_1$.
If we divide the score above by $\alpha_1$, we get the interval score
\begin{align}\label{eq:tgscore}
\theta_2-\theta_1+\frac{1_{x<\theta_1}}{\alpha_1}(\theta_1-x)+\frac{1_{x>\theta_2}}{\alpha_1}(x-\theta_2)
\end{align}
due to \cite{wink:1972}.
The above derivation uses $1_{x\le\theta_1}(\theta_1-x)=1_{x <\theta_1}(\theta_1-x)$.
Equation~\eqref{eq:tgscore} is equation (43) of \cite{gnei:raft:2007}.  
Their $\alpha$ is our $1-\alpha$, so for the default
coverage setting, our $1/\alpha_1=40$ is their $2/\alpha=40$.

When scoring confidence intervals $[L,U]$ for the mean $\mu$, if we reverse the
notion of what is random, as if $\mu$ were random with estimated quantiles
$L$ and $U$, then the Winkler score~\eqref{eq:tgscore} yields
\begin{align}\label{eq:tgciscore}
U-L+\frac{1_{\mu<L}}{\alpha_1}(L-\mu)+\frac{1_{\mu>U}}{\alpha_1}(\mu-U).
\end{align}
This measure combines the length $U-L$ of the interval
with some non-coverage penalties that grow with the distance between $\mu$
and the interval $[L,U]$. 
This measure was proposed in the masters thesis \cite{hofe:2022} 
and used in \cite{hofe:held:2022} to score confidence
intervals for the mean $\pi$ of a Bernoulli distribution. They find
that Wald intervals perform badly, and for $\pi$ not near
$0$ or $1$, the best results were from a Wilson confidence
interval or a Bayesian credible interval using a uniform prior.

Scoring a confidence interval the way we would score a prediction
interoval involves reversing the role of what is fixed with what is random.
Statisticians are very used to this when alternating between frequentist
and Bayesian viewpoints of confidence/credible intervals. So it is
surprising that the proposal in \cite{hofe:held:2022} is so recent.

An upper confidence interval function would be a statistic
$T_n(x_1,\dots,x_n)$ with the property that
$$
\Pr\bigl( \e(x_1) \le T_n(x_1,\dots,x_n)\bigr)=\alpha
$$
for $x_i\simiid F$.  By the result of \cite{baha:sava:1956} we know
that no non-trivial such $T_n$ exists.  
Perhaps we should not expect to elicit a statistical quantity that does not exist.
Some quantities that are closely related to confidence intervals
are known to not be elicitable.
\cite{fron:kash:2014} call the interval $(a,b)$ a 90\% confidence
interval for $x$ if $\int_a^bp(x)\rd x=0.9$.  That does not correspond
to our usage of the term `confidence interval'.  It is a prediction
interval.  They go on to show that the entire set of such pairs
$(a,b)$ is not elicitable, nor is the pair with the smallest
width $\omega=b-a$, nor is $\omega$ itself elicitable.
For more in this direction, see \cite{breh:gnei:2021}.
From here on, we assume that there is no consistent score function
for non-trivial  confidence intervals.

\subsection{Decision theory}\label{sec:decisiontheory}

There has been some work on decision theoretic frameworks for confidence intervals
and set estimation more generally.  Good starting points for that literature are
\cite{case:hwan:1991} and \cite{case:hwan:robe:1994}.
Those works are for a parametric setting, which is understandable due to the
non-existence of fully nonparametric confidence intervals for the mean.
The most prominent loss functions take a form equivalent to
\begin{align}\label{eq:usualloss}
\lambda\times (U-L)-1\{L\le\mu\le U\}
\end{align}
for some $\lambda>0$.  Authors vary in whether the factor
$\lambda$ is applied to the interval length or the coverage indicator.
Then the risk is
\begin{align}\label{eq:cirisk}
\lambda\times\e( U-L) -\Pr(\mu\not\in[L,U]),
\end{align}
trading off expected length and coverage probability.
The great difficulty is in choosing the value for $\lambda$.
If we are comparing methods that achieve $\Pr(\mu\not\in[L,U])=\alpha$
then any $\lambda>0$ leads us to prefer shorter mean length.  When
we have to compare two or more imperfectly calibrated methods,
as we must for ACIs, then it is very hard to choose the constant~$\lambda$.

A more serious difficulty is that if we fix $\lambda$ and then minimize
risk, we are also effectively choosing the level $\alpha$ for our interval.
When we are reporting a confidence interval or plotting error bars
we want to directly specify and communicate the nominal level.
It would not be satisfactory for the coverage level $\alpha$ 
corresponding to some $\lambda$
to be different from our nominal $\alpha$. It would be even worse to communicate
intervals whose nominal level was unknown.

\cite{case:hwan:robe:1994} describe a paradox told to them by James
Berger.  When using the loss~\eqref{eq:usualloss} on Gaussian
data, the usual Student's $t$ interval can be dominated
by a strategy that truncates the interval to length $0$
when the sample standard deviation $s$ is large enough.
The paradox can be removed through the use of an increasing scale function
$S:[0,\infty)\to[0,1]$ and then  using
\begin{align}\label{eq:newloss1}
S( U-L)-1\{L\le\mu\le U\},
\end{align}
where $S$ could incorporate a scaling factor like $\lambda$.

For ACIs, the coverage error is typically $O(1/n)$ while
the mean length is typically $O(1/\sqrt{n})$.
Then if $\lambda$ is reasonable for one value of $n$
it is not suitable for other values of $n$.
We can make the choice of $\lambda$ less dependent on $n$
by either penalizing the mean squared length, or by using
\begin{align}\label{eq:newaciloss}
\lambda\frac{U-L}{\sqrt{n}\sigma}-1\{L\le\mu\le U\}.
\end{align}
Scaling by $\sigma$ makes the performance invariant to
rescaling the data $x_i$ to $cx_i$ for $c>0$ when, as is usual,
coverage is also invariant to such rescaling.
In numerical examples, we know $\sigma$ and can then
use it in comparisons of methods.

It is not obvious that an interval becomes better
when its coverage level exceeds the nominal level.
Over-estimating the coverage of the interval $[L,U]$
necessarily under-estimates the coverage of $(-\infty,L)\cup(U,\infty)$.
For 95\% coverage, criteria like
\begin{align}\label{eq:ciriskalt1}
\lambda\frac{\e( (U-L)^2)^{1/2}}{\sigma\sqrt{n}} -|\Pr(\mu\not\in[L,U])-0.95|,
\end{align}
or
\begin{align}\label{eq:ciriskalt2}
\lambda\frac{\e( (U-L)^2)^{1/2}}{\sigma\sqrt{n}} -\min\bigl(0.95,\Pr(\mu\not\in[L,U])\bigr)
\end{align}
explicitly account for the target coverage level, unlike~\eqref{eq:cirisk}.
Using the root mean squared length (RMSL) in these criteria
penalizes methods that produce highly variable interval lengths.

The measure in~\eqref{eq:ciriskalt1} penalizes both under-coverage
and over-coverage.
The measure in~\eqref{eq:ciriskalt2} does not explicitly
penalize over-coverage, but does not also reward it.
Confidence interval constructions usually have a parameter
such as $t_{(n-1)}^\alpha$ in Student's $t$ that can be
raised or lowered to trade off length and coverage.
Methods that over-cover are then methods that might otherwise
have had nominal coverage at lower length, e.g., smaller $t^\alpha$. 
Their over-coverage
is then already penalized by that extra interval length.
Neither of the above criteria
can be written as values of expected loss functions. In that sense
they are only `quasi-risks', not risks.
In any of the criteria~\eqref{eq:cirisk}, \eqref{eq:ciriskalt1} and \eqref{eq:ciriskalt2}
it remains unclear  how one should choose~$\lambda$.

\subsection{The $F_1$ score}

In binary classification problems, some methods may have better
precision while others have better recall.  These two desiderata conflict in a 
very similar way to how the criteria for ACIs do.  
The $F_1$ score is commonly used to combine them into one criterion.
The context is that of predicting a binary label $y\in\{0,1\}$ where items with $y=1$
are rare and interesting to discover, while $y=0$ is the norm.
This is similar to the problem of choosing a confidence interval
where we must trade off two competing criteria.
And yet, the $F$-score described next is a very popular, though not universally popular,
way to do it.
An early context was information retrieval \citep{vanr:1979} where $y=1$
describes a relevant document and earlier uses were in biology \citep{dice:1945,sore:1948}.

Given estimates $\hat y\in\{0,1\}$ the precision 
is $p = \Pr(y=1\giv\hat y=1)$ and the recall is
$r = \Pr( \hat y=1\giv y=1)$.
These are typically estimated using empirical probabilities
over a set of known $(y,\hat y)$ pairs.
The $F_1$ score is the harmonic mean
$$
F_1 = \frac1{(1/p+1/r)/2} = \frac{2pr}{p+r}
$$
and it is widely used as a single measure of classifier accuracy.
For some debates about its suitability see \cite{hand:chri:kiri:2021}
and for a listing of many alternatives, see \cite{powe:2020}.
The $F_1$ score is preferred to the overall accuracy $\Pr(\hat y=y)$ because the latter can be
very high for the completely uninformative rule with $\Pr(\hat y=0)=1$.
The more general $F_\beta$ score given by
$$
F_\beta = \frac{\beta^2+1}{\beta^2/r+1/p}
$$
puts more emphasis on recall when $\beta>1$ and more
emphasis on precision when $\beta <1$.

We can draw an analogy between binary classification and confidence
intervals.  The true value $\mu$ is the one real number that we wish
to recover.  That makes coverage probability the analogue of recall.
Taking $\tilde r = \max( \Pr( L\le\mu\le U),0.95)$ allows us to
incorporate a specific nominal coverage level, which we were not
able to do with the Winkler score.  This criterion is designed to
reward absolute coverage (up to a limit) instead of penalizing
coverage error.

For precision, a shorter interval corresponds to greater precision.
We want something proportional to $1/\rmsl$. We choose
$\tilde p = \sigma/(\sqrt{n}\,\rmsl)$. The scaling by $\sigma$
makes the criterion invariant to rescaling $x_i$ and this is appropriate
because $\tilde r$ is invariant. Where coverage error decreased
as $1/n$, absolute coverage approaches the nominal level and
our scaling by $\sqrt{n}$ makes $\tilde p=O_p(1)$.
With these choices
\begin{align}\label{eq:tildef}
\tilde F_\beta = \frac{\beta^2+1}{\beta^2/\max(\Pr(L\le\mu\le U),0.95)+\sqrt{n}\rmsl/\sigma}.
\end{align}
To get a $\tilde F_\beta$ score suitable for ACIs we would need a consensus
choice for $\beta$.

\subsection{Some other criteria}

\cite{ci4rqmc} prioritized coverage over length.  For a nominal 95\% coverage
any method delivering below 94\% coverage would be considered a failure.
If a method did not cover the true mean at least 927 times out of
1000, then it was deemed to have failed.  A false positive would happen with
less than 4\% probability, by the binomial distribution.
That large study of randomized quasi-Monte Carlo integration featured
2400 tasks from 6 integrands, 5 RQMC methods, 4 integral dimensions,
5 RQMC sample sizes and 4 choices for the number of RQMC replicates to use.
The percentile bootstrap failed 1698 times, the bootstrap $t$ failed
81 times, and the plain Student's $t$ interval failed only 3 times.

The RQMC setting is special.  Recent literature is showing that
the integral estimates are nearly symetrically distributed with
large kurtosis. For one kind of RQMC, the diverging kurtosis follows
from results in \cite{superpolyone} while the modest skewness has
been studied in \cite{lowskewness}.
It is well known that the standard Student ACIs attain
high coverage, often over the nominal level, 
for data that come from a symmetric distribution with
very heavy tails.  See \cite{mill:1986}, \cite{efro:1969} and \cite{cres:1980}.

\cite{owen:smallci}
compared nonparametric ACIs for the
mean. Portions of that work appeared in \cite{elsmall}.
In one of the comparisons each method chose an ACI
and then each other method selected its ACI
using the same length that the first method chose.  That allows us
to compare methods by a measure of length adjusted coverage.
In that framework, the empirical likelihood intervals
generally did best.  They were not however as well calibrated as bootstrap $t$.
Improving the calibration of empirical likelihood intervals is outside
the scope of this article.

\cite{wolf:1950} described an interesting way to modify the length
of a confidence interval in a loss function.
When an interval does not cover $\mu$, he extends it until
it does and then uses 
the length of the resulting interval $[\min(U,\mu),\max(L,\mu)]$.
In some numerical work this did not make much difference.
That can be explained by non-coverage being relatively rare
and the distance between the interval and $\mu$ being $O_p(n^{-1/2})$.

\subsection{Illustration}

To illustrate the difficulties in quantifying ACI quality, consider ACIs for the mean
of the $\dexp(1)$ distribution based on $x_1,\dots,x_{10}\simiid\dexp(1)$.
This distribution is only a  moderately difficult one from
Section~\ref{sec:testcases}.
The sample size $n=10$ is near the middle of the range $2\le n\le20$
considered there.

The true mean is $\mu=1$ and the mean coverage of any given ACI is now a $10$
dimensional integral.  We can approximate that integral using
$10{,}000$ vectors $X\in\real^{10}$ with IID $\dexp(1)$ components.
The results are in Table~\ref{tab:coverexpo10} for bootstrap
methods based on $B=2000$ resamples, sorted in decreasing
order by coverage.
If we go by coverage and average length, then all of the methods
are on the Pareto frontier, except for the half sampling
bootstrap $t$, which has greater mean length and lower coverage 
than the Poisson bootstrap $t$.

\begin{table}[t]\centering
\begin{tabular}{lrrrrrr}
\toprule
  Method & Cover\hspace*{.75mm} & Length\hspace*{-2mm} & Length$^2$\hspace*{-2.5mm} & Win95\hspace*{-1mm} & Win99\hspace*{-1mm} & $\tilde F_1$\hspace*{2mm}\\
\midrule
U$\{0,2\}$ & 0.9881 & 4.70 & 44.58 & 4.77 & 5.04 & 0.090\\
Poisson & 0.9612 & 2.26 & 8.07 & 2.49 & 3.40 &0.200\\
Half Sample & 0.9577 & 2.37 & 9.37 & 2.62 & 3.60 &0.187 \\
Multinomial & 0.9415 & 1.87 & 5.40 & 2.22 & 3.62 &0.238\\
Power & 0.9400 & 1.62 & 3.32 & 1.98 & 3.40 & 0.293 \\
Beta($1/2,3/2$) & 0.9311 & 1.52 & 2.90 & 1.94 & 3.59 & 0.311\\
Student's t & 0.9002 & 1.33 & 2.07 & 1.96 & 4.50 & 0.357 \\
Exponential & 0.8935 & 1.23 & 1.84 & 1.90 & 4.58 & 0.374\\
\bca & 0.8772 & 1.17 & 1.64 & 1.92 & 4.95 & 0.392\\
Mammen & 0.8681 & 1.12 & 1.43 & 1.96 & 5.36 &0.413\\
Lognormal & 0.8405 & 1.01 & 1.20 & 2.04 & 6.17&0.443 \\
\bottomrule
\end{tabular}
\caption{\label{tab:coverexpo10}
  This table describes $95$\% ACIs for a sample of $n=10$
  draws from $\dexp(1)$. All the methods are weighted
  bootstrap $t$ except for \bca\ and Student's $t$.
  The columns are mean coverage, mean interval
  length, mean squared interval length, two Winkler scores, one at
  level  $0.95$   and one at level $0.99$, and $\tilde F_1$ from~\eqref{eq:tildef}.
}
\end{table}

If we use the Winkler score for level $0.95$
then the best score comes from a bootstrap $t$ using
Rubin's Bayesian bootstrap, i.e., exponential weights.
It does not seem reasonable that a method with
$89.35$\% coverage should be best when other methods
attain 93\% to 96\% coverage.  Those other methods
use longer intervals, but they do not appear to have drastically
greater length.

Sometimes experts making a prediction might make a prediction
that is more cautious than their true belief if the score function
is not proper. In a setting with no proper score function like
ACIs appear to be, we might judge that the 95\% Winkler
score is too lenient on coverage (versus length) and compensate
by scoring nominally 95\% ACIs by a higher coverage goal
such as 99\%.  By that measure we would favor the bootstrap $t$
with power weights.  It raises the coverage from just over 89\%
to 94\% while increasing the mean length by a factor
of $1.62/1.23\doteq 1.32$ and the mean squared length
by $3.32/1.84\doteq 1.80$.

Poisson weights attain essentially the
same score as power weights do here.  They do so by
getting credit for coverage that goes beyond even 95\%.
While it may or may not be reasonable to penalize a method for
covering more than the nominal level, it does not seem
reasonable to credit it for covering more than the nominal
level.  This breaks the Winkler score tie in favor of the power
bootstrap $t$.

If we insist on at least 95\% coverage and then minimize
mean or mean squared length subject to that constraint, then Poisson weights
did best for this case. 
If we go by the largest values of $\tilde F_1$ then the favored methods
would be the ones that most severely under-cover $\mu$.
In other words, we are not fortunate enough to have a setting where $\beta=1$
provides a reasonable default.  For other values of $\beta$ $\tilde F_\beta$ generally
favors either the longest or the shortest measures.  It is essentially picking one
criterion and neglecting the other.

The author's subjective judgement of the results in Table~\ref{tab:coverexpo10}
is as follows.  The multinomial, power and beta bootstrap $t$ methods are
all reasonable choices attaining nearly nominal coverage without
excessive length.
The bottom five methods, from Student's $t$ to the lognormal
bootstrap $t$ have coverage levels that are unacceptably low 
by comparison.  While the top three methods cover even more
than the target level, they do so with unreasonable extra length.

This case also illustrates how some bootstrap $t$ methods,
such as those using lognormal or Mammen weights
can cover even less often than \bca\ does.  At the same time,
three of the bootstrap $t$ methods cover even more than
the usual multinomial bootstrap $t$ does.  Those are the
ones that have even greater probabilities of assigning
a weight of zero to an observation.

\section{Test cases}\label{sec:testcases}

We can examine some specific distributions to find the resulting coverage
probabilities and interval lengths for different bootstrap methods.
We begin with the asymptotic theory from \cite{hall:1988}
because that lets us select distributions to provide quite different tests
of the ACIs.
Hall's asymptotic theory is based on moments with assumed finite
eighth moments and a Cram\'er condition. So we include some
distributions with unbounded fourth moment and also
a discrete distribution for which the Cram\'er condition
does not hold.

\subsection{Asymptotic coverage}\label{sec:hallasymptotic}

Table 1 of \cite{hall:1988} provides some asymptotic coverage formulas
for bootstrap ACIs for the mean of $n$ IID observations
as the sample size $n\to\infty$.
\cite{likehall} added a formula for Student's $t$ method.
The method of \cite{hall:1988} for Normal theory
used $\hat\sigma$ and quantiles of $\dnorm(0,1)$
instead of $s$ and $t_{(n-1)}$ quantiles.

The coverage errors for 95\% confidence intervals
are $\Pr( L \le\mu\le U)-0.95$, so negative values
indicate undercoverage.  The asymptotic formulas
involve the skewness $\gamma = \e( (x-\mu)^3)/\sigma^3$
and the kurtosis $\kappa = \e((x-\mu)^4)/\sigma^4-3$.
The coverage errors for two-sided 95\% confidence intervals are
\begin{align*}
\text{Normal theory:} &\quad(1/n) \varphi(z^{0.975})\bigl[\phantom{-}0.14\kappa-2.12\gamma^2-3.35\bigr] +O(1/{n^2}),\\
\text{Student's $t$:} &\quad(1/n) \varphi(z^{0.975})\bigl[\phantom{-}0.14\kappa-2.12\gamma^2\,\,\,\phantom{-3.35}\bigr] +O(1/{n^2}),\\
\text{Percentile:} &\quad(1/n)\varphi(z^{0.975})\bigl[-0.72\kappa -0.37\gamma^2-3.35\bigr] +O(1/{n^2}),\\
\text{Bootstrap $t$:} &\quad(1/n)\varphi(z^{0.975})\bigl[-2.84\kappa+4.25\gamma^2\,\,\,\phantom{-3.35}\bigr] + O(1/{n^2}),\quad\text{and}\\
\text{\bca{}:} &\quad(1/n)\varphi(z^{0.975})\bigl[-2.63\kappa+3.11\gamma^2-3.35\bigr] + O(1/{n^2}),
\end{align*}
where $\varphi$ is the $\dnorm(0,1)$ probability density function.

The error expressions above help to explain why the bootstrap $t$ is often
so effective. It has the largest positive coefficient for $\gamma^2$ 
of all the methods in Table 1 of \cite{hall:1988} and it is the only bootstrap
method with a zero intercept instead of $-3.35$.  

Plugging $n=18$ into Hall's formula gives very
close the coverage error averaged over
$16\leqslant n\leqslant 20$ in Table 6 of \cite{owen:smallci}
for all distributions but one.
That one was the lognormal distribution where the formula
gave a coverage error so low that it corresponded 
to negative coverage for the ACI.

Some of the gaps in Hall's analysis can be filled in by
parts C, D and E of Theorem 1 of \cite{sing:1981} which assume
that $\e(|x-\mu|^3)<\infty$. Then coverage errors are $o(n^{-1/2})$
in the non-lattice case and $O(n^{-1/2})$ in the lattice case.

\subsection{Computations}
Each sampling distribution was represented in $R=10{,}000$
data sets of sizes $2\le n\le 20$.
All of the ACI methods were given the same $10{,}000$
samples of each size so they all received the
same problem difficulty. Furthermore the 
samples of size $n+1$ always contained the points
from the samples of size $n$.  That technique
gives more precise estimates of how length or coverage
vary with $n$.  For $R=10{,}000$ replicates, and a method
with coverage probability $0.95$, the observed number of
of coverages has $0.005$ quantile 9443 and $0.995$ quantile 9555.
Then the coverage proportion has about $99.9$\% probability
of being within $0.57$ percentage points of nominal.  For
a method with coverage $80$\% comparable intervals are
roughly $0.80\pm 0.01$.
Every bootstrap was based on $B=2000$ resamplings. 

Since there is no measure to elicit overall ACI quality
and no generally accepted way to trade off length versus
coverage, we plot a measure of interval length
versus coverage in order to visualize the Pareto
frontier.  Based on the discussion in Section~\ref{sec:decisiontheory},
interval length is measured by the root mean squared
length (RMSL) divided by $\sqrt{n}$ so that it approaches
zero as $O(1/n)$ which is the same rate that
two-sided coverage typically approaches the nominal level.
We also divide the RMSL by $\sigma=\var(x)^{1/2}$ to make
the results for different distributions more nearly comparable.

Well calibrated methods will show a trajectory descending
towards a reference point with coverage $0.95$ and RMSL $0$.
That trajectory is asyptotically linear in the plot coordinates
when coverage error and RMSL are  asymptotically proportional to $1/n$ and $1/\sqrt{n}$, respectively.  Well calibrated methods approach the reference
point from above.  Methods with severe under-coverage  approach the reference point from above and to the left.  
RMSL$/\sqrt{n}$ is always decreasing with $n$ in the ranges shown. For the most
severe under-coverage the trajectories approach the reference almost from the
left.  For some settings
the coverage is not monotone in $n$ over the range explored.  There can be
extreme non-monotonicities for $n\le 6$.

The methods can be clustered into a few groups.
The bootstrap $t$ methods with the greatest values of $\Pr(v_i=0)$
are those with double-or-nothing, half sampling and Poisson
weights. Those methods generally have the longest intervals
and the highest coverage. They also have the most frequent occurences
of infinite length intervals. They are plotted in red/brown/orange
colors, with the more reddish color for the longer intervals.  
The second group is made up of the usual multinomial bootstrap
$t$ as well as the ones with power and beta distributed weights have
lengths and coverage levels lower than the first group.
They are colored by shades of green, with brighter green for the methods with generally longer intervals.
A third group with even shorter intervals and generally very low
coverage is made up of the \bca\ along with bootstrap $t$ with Mammen weights
and  exponential (i.e., Bayesian bootstrap) weights. 
They are plotted in shades of blue and the ordering among them varies with the setting.
ACIs based on Student's $t$ distribution do not fit cleanly into
any of those three groups.  They appear in different 
groups depending on the distribution of $x$, and are shown
in black.  The bootstrap $t$
using lognormal weights have extremely short intervals and
low coverage.  They are included only to show that it is possible to
have weighted bootstrap $t$ intervals with lower coverage (and length)
than \bca\ intervals have. They are shown in gray.

\subsection{Benign/Gaussian data}

A very benign setting for nonparametric ACIs arises when the $x_i$
have a unimodal density function with at most mild skewness and
light tails.  As an archetypal example of such a distribution, 
we consider $x_i\sim \dnorm(0,1)$.  Of course, in practice a user
would not know that the data were from such a parametric model.
We want to see  how methods do on this case because a method
that does poorly in even such a simple setting is not a good choice
overall.

\begin{figure}[t]
\centering
\includegraphics[width=\linewidth]{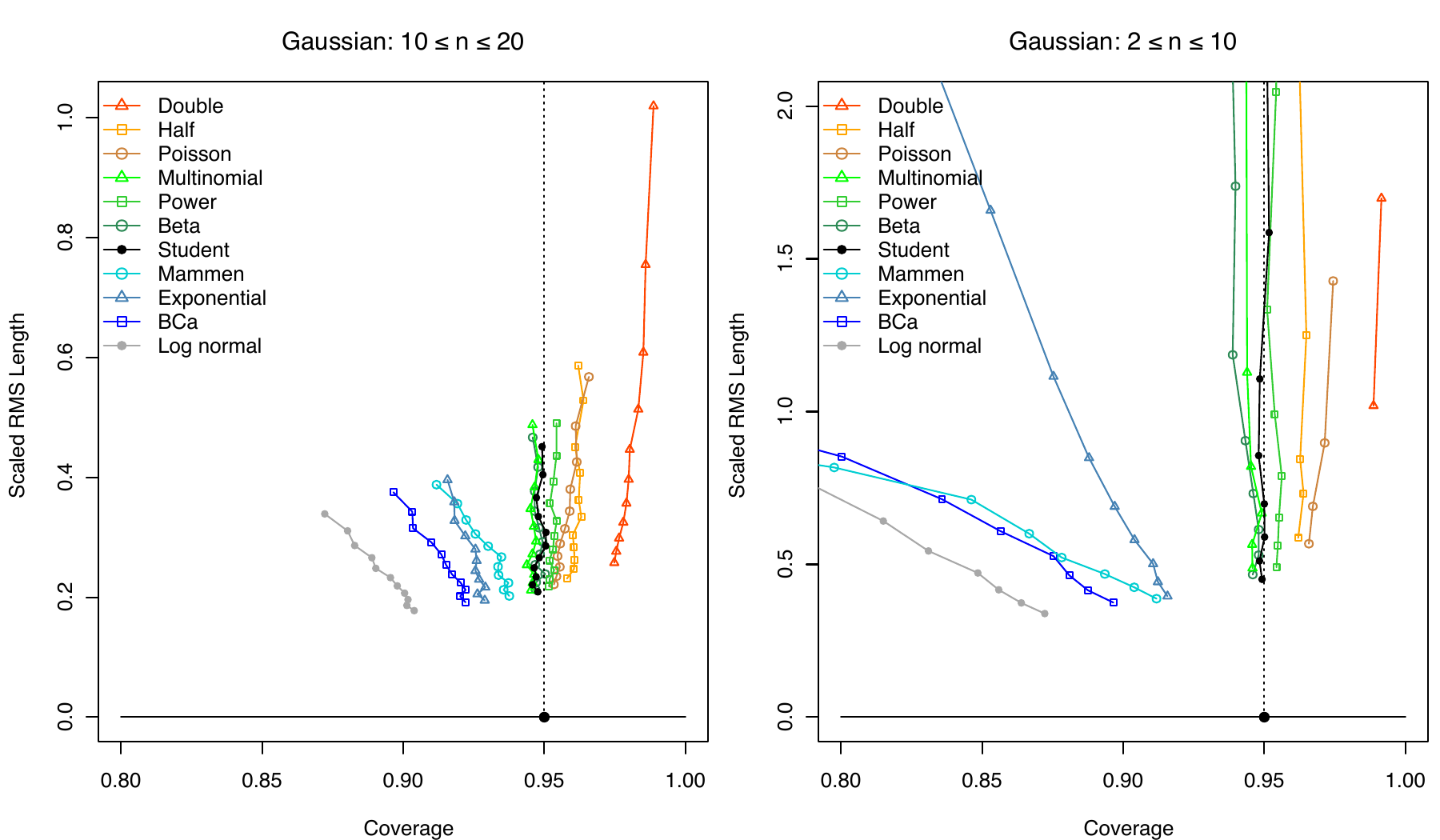}
\caption{\label{fig:sgaus}
This figure shows  RMSL$/\sqrt{n}\sigma$ versus coverage
when $x_i\sim \dnorm(0,1)$. 
The left panel has $10\le n\le20$ and the right panel
has $2\le n\le 10$.  The scaled RMSL decreases with $n$. There is a reference
point for nominal coverage and  $0$ length.  There is a dotted reference
line at the nominal coverage level.
}\end{figure}

Figure~\ref{fig:sgaus} shows estimates of RMSL and coverage
for Gaussian data.  There are two views, one for $10\le n\le 20$
and one for $2\le n\le 10$.
The methods cluster according to the
earlier described groupings.
The double-or-nothing bootstrap has by far the longest
intervals and it covers far more than the nominal amount.
Mammen, exponential and \bca\ intervals are all very short
and under-cover significantly, even on this easy case. The multinomial, power and
beta bootstrap $t$ intervals are intermediate with approximately
nominal coverage and are close to the Student $t$ method
which has exact coverage for this problem.

The bottom of the curves in the left panel shows that for $n=20$
there is a near perfect monotone relationship between the coverage
level and the RMSL.  The one exact method, Student's $t$ interval,
has a very slightly better RMSL than the multinomal and Beta
bootstrap $t$ methods with nearby coverage levels.

From the right panel we can see that the double-or-nothing
bootstrap only had a finite RMSL for $n\ge9$. 
This simply means that at least one of the $10{,}000$
ACIs it produced had infinite length.
An analysis in Appendix~\ref{sec:conventions} shows that
for distinct $x_i$ and $B\to\infty$ that a finite length
ACI will happen for $n\ge9$ but not necessarily for $n\le 8$.
The Poisson ACIs were only finite for $n\ge7$ which also
matches the cutoff in Appendix~\ref{sec:conventions}.
When $n$ is just barely large enough to assure a finite
ACI length, then the RMSL can be very heavy-tailed.
The plot does not show every finite RMSL in order
to better show the bulk of them.

\begin{table}[b]
\centering
\begin{tabular}{lccccccc}
\toprule
& \multicolumn{2}{c}{$n=2$}& \multicolumn{2}{c}{$n=3$}& \multicolumn{2}{c}{$n=4$}\\
\midrule
Method &   Cover & RMSL&  Cover & RMSL & Cover& RMSL\\
\midrule
Power & $0.9650$ & $26.55$ & $0.9563$ & $7.66$ & $0.9542$ & $4.09$\\
Student's $t$ & $0.9493$ & $17.94$ & $0.9501$ & $4.98$ & $0.9517$ & $3.17$\\
Beta & $0.9389$ & $15.31$ & $0.9356$ & $5.76$ & $0.9398$ & $3.48$\\
\bottomrule
\end{tabular}
\caption{\label{tab:lowgaus}
Coverage and unscaled RMSL for small sample sizes on Gaussian data.
}
\end{table}

The power and beta bootstrap $t$ methods do very well
for the Gaussian data even with tiny $n$.
Table~\ref{tab:lowgaus} shows the coverage and RMSLs
for $2\le n\le 4$ for the power and beta bootstrap $t$
methods along with the value for Student's $t$.
The coverage errors are much closer to nominal
there than we see for much larger $n$ on 
more difficult distributions.

\subsection{Skewed/exponential data}

Skewness is one of the main reasons to switch from Student's $t$
distribution to some kind of bootstrap distribution.
Our archetypal example for mildly skewed data has $x_i\sim\dexp(1)$.
It has skewness $2$.
There is no universal rule for what amount of skewness
is mild and what is severe.  Some authors consider skewness
in $[-1,1]$ to be mild and skewness $2$ to already be severe.
Others consider skewness of $2$ to not be extreme.
Where \cite{hest:2015} recommends 
$n\ge 4815$ for Student's $t$ test he finds
that the bootstrap $t$ with $n\ge101$ meets his stringent conditions.

\begin{figure}[t]
\centering \includegraphics[width=\linewidth]{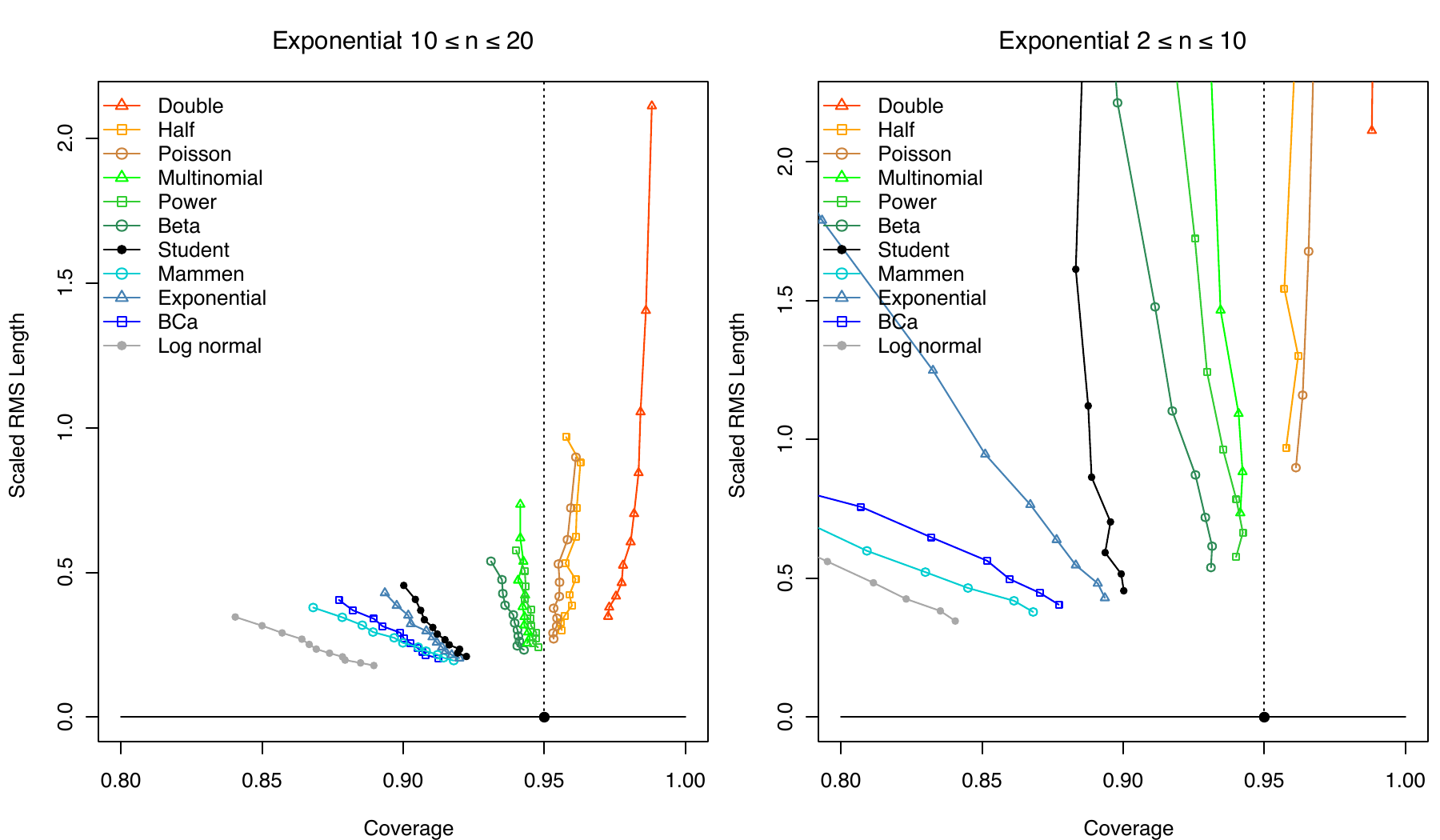}
\caption{\label{fig:sexpo}
Same as Figure~\ref{fig:sgaus}
except that $x_i\sim \dexp(1)$.
}\end{figure}

Figure~\ref{fig:sexpo} shows results for $x\sim \dexp(1)$.
Now Student's $t$ clusters with \bca\
and the Mammen and exponential bootstraps, instead  of
with the beta, power and multinomial bootstrap $t$. 
Student's $t$ and \bca\  under-cover significantly.
The multinomial, power and beta bootstraps slightly under-cover.
The power version has comparable coverage to the multinomial
over $10\le n\le20$ with somewhat smaller RMSL.

Half sampling, Poisson and double-or-nothing bootstraps
over cover and have much greater length.
The double-or-nothing bootstrap has infinite RMSL for $n\le 8$.

\subsection{Heavy-tailed symmetric/$t_{(4)}$ data}

As noted above  the standard Student ACIs attain
high coverage on symmetric heavy tailed data.
Our archetypal example for this case is $x_i\sim t_{(4)}$.
This distribution is symmetric with skewness zero and infinite kurtosis.
As a result the asymptotic coverage results from 
\cite{hall:1988} do not apply for it, but Theorem 1C  of
\cite{sing:1981} shows that coverage errors are $o(n^{-1/2})$
for the bootstrap $t$ because $\e( |x|^3)<\infty$ and the distribution is non-lattice.

\begin{figure}[t]
\centering
\includegraphics[width=\linewidth]{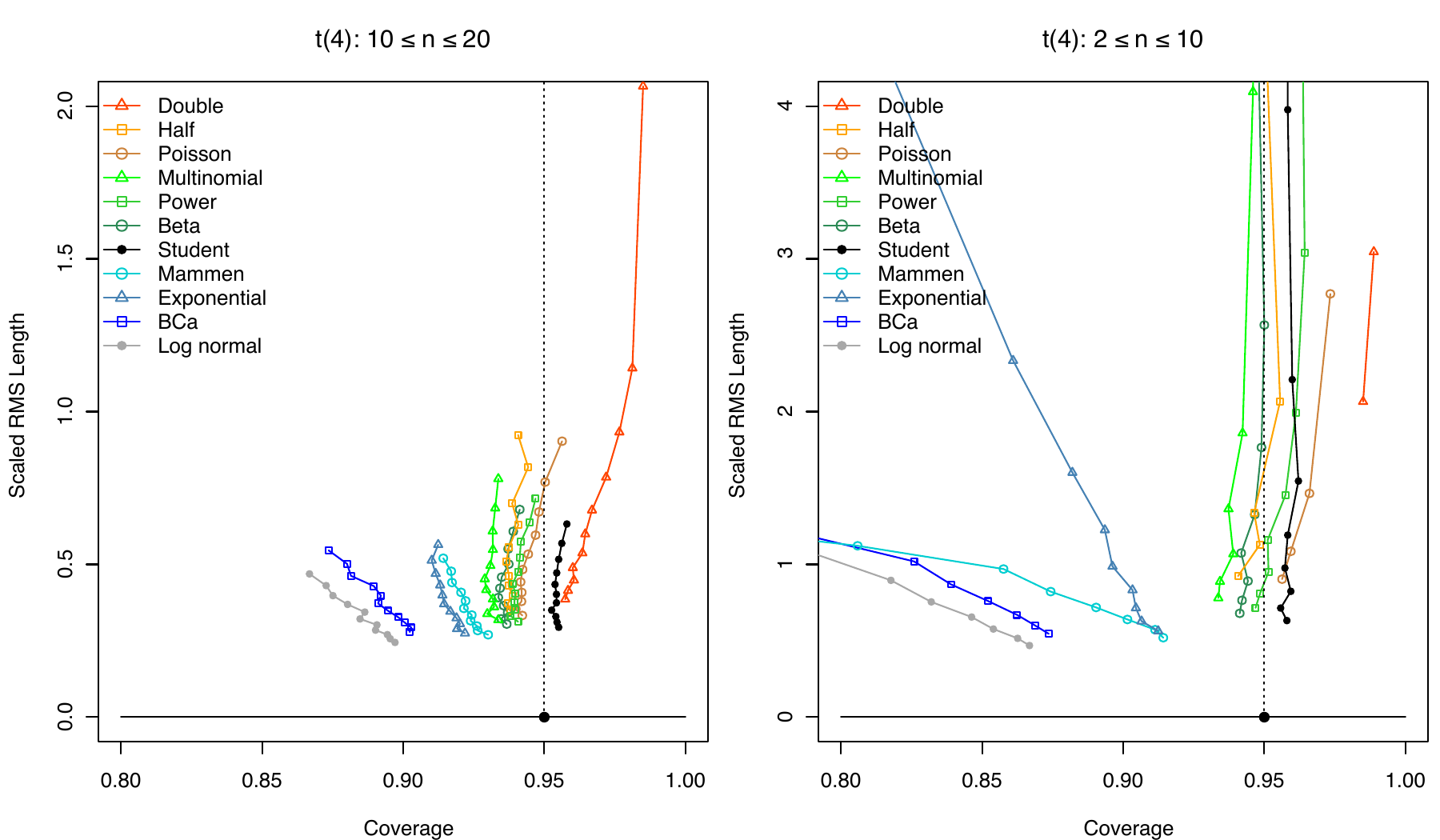}
\caption{\label{fig:stfor}
Same as Figure~\ref{fig:sgaus} except that 
$x_i\sim t_{(4)}$.
}\end{figure}

Student's $t$ ACIs are the best by far here in line
with well known theory.  All of the bootstrap $t$
method under-cover except for the double-or-nothing
version which has infinite RMSL for $n\le 8$.
\bca\ barely attains over 90\% coverage by $n=20$.

\subsection{Light-tailed symmetric/$\dunif(0,1)$ data}

Here we consider light-tailed symmetric data, using $\dunif(0,1)$
as the example. Figure~\ref{fig:sunif} shows the results.  Student's $t$
does best which is not surprising for symmetric data.  For $n\ge10$,
most methods significantly over-cover the true mean. The Bayesian
bootstrap (exponential weights) and \bca\ under-cover but come
quite close to the nominal level.

\begin{figure}[t]
\centering
\includegraphics[width=\linewidth]{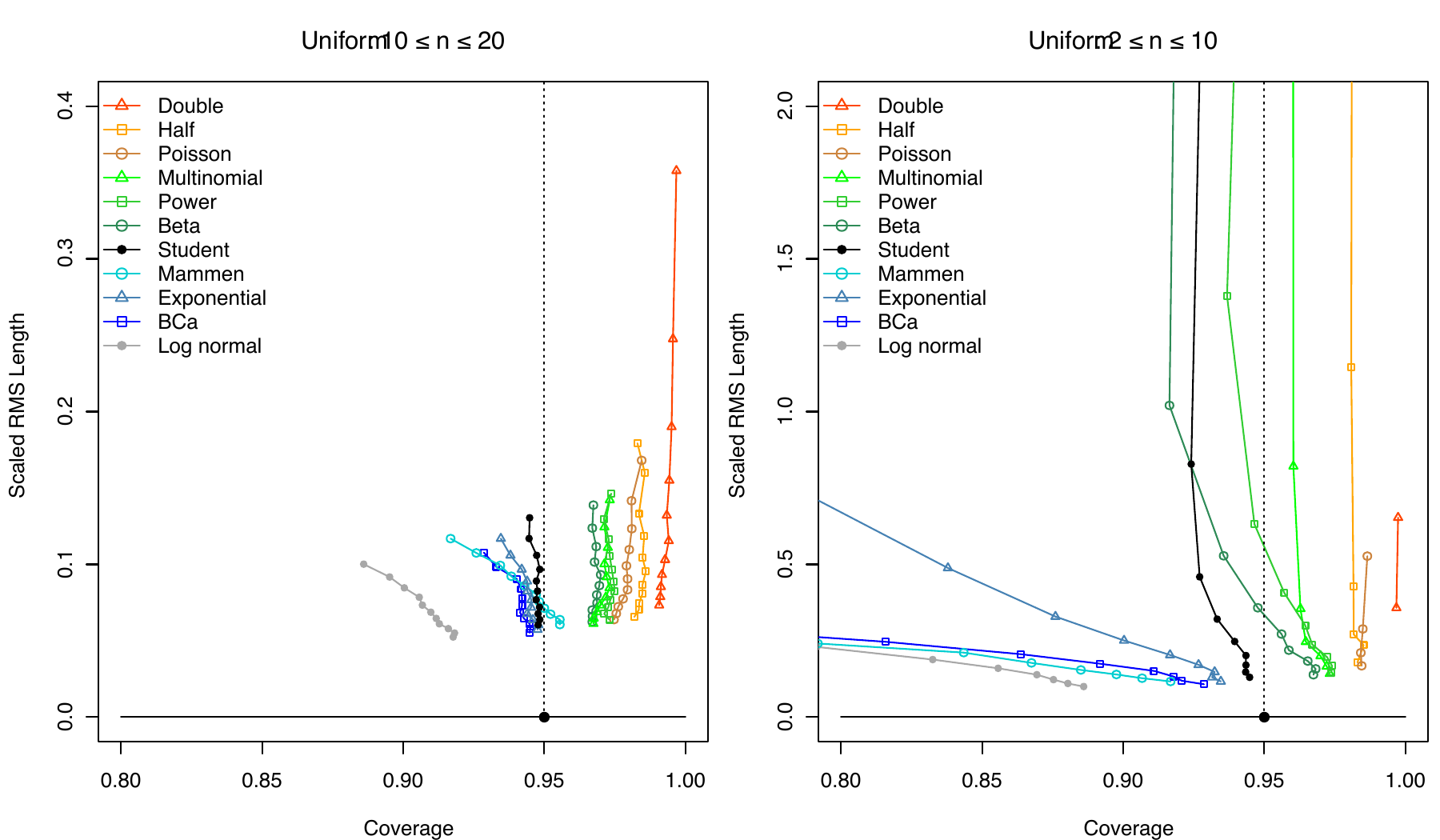}
\caption{\label{fig:sunif}
Same as Figure~\ref{fig:sgaus} except that 
$x_i\sim \dunif(0,1)$.
}\end{figure}

\subsection{Severe/lognormal setting}

A severe setting has a very large kurtosis but is also not symmetric
enough for Student's $t$ ACIs to do well.  For this case, we consider
the lognormal distribution $x_i\sim \exp(\dnorm(0,1))$.
None of the bootstrap methods in \cite{owen:smallci} were able to get much more
than 90\% coverage by $n=20$ for this distribution when the nominal
rate was 95\%. The predicted coverage using the formulas of \cite{hall:1988}
in Section~\ref{sec:hallasymptotic}
was actually negative because of the extremely large kurtosis.
This distribution has 
\begin{align*}
\kappa &= \exp(4)+2\exp(3)+3\exp(2)-6\doteq 110.93
\quad\text{and}\\
\gamma &= (\exp(1)+2)\sqrt{\exp(1)-1}\doteq 6.18.
\end{align*}

\begin{figure}[t]
\centering
\includegraphics[width=\linewidth]{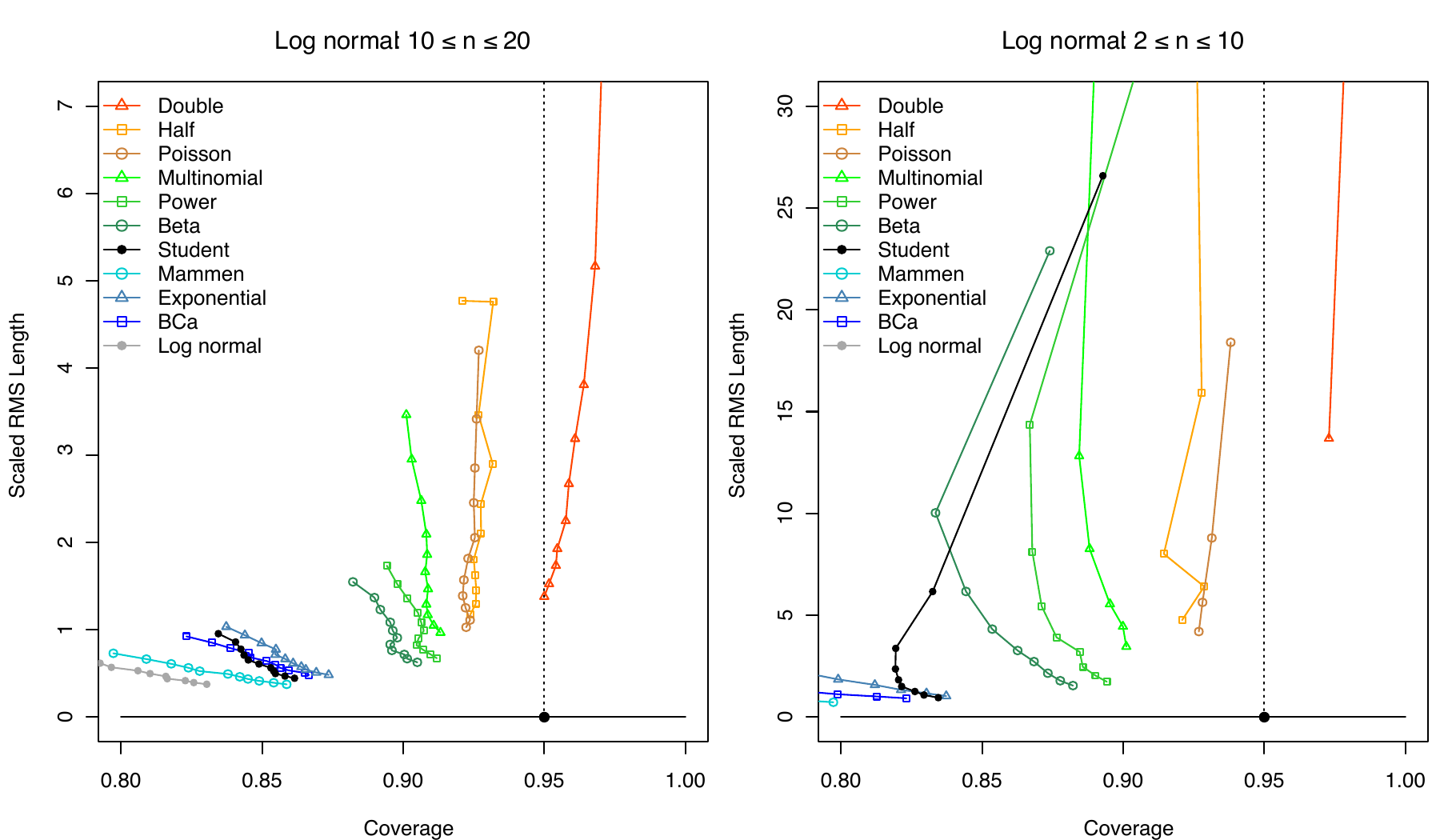}
\caption{\label{fig:slnor}
Same as Figure~\ref{fig:sgaus} except that $x_i\sim \exp(\dnorm(0,1))$.
}\end{figure}

This lognormal data is very heavily skewed and has an enormous kurtosis.
For $10\le n\le20$,
the double-or-nothing bootstrap has somewhat over nominal coverage.  
It is extremely hard for any of the other methods to cover well.
Student's $t$ ACIs severely undercover and are clustered
with \bca\ and the Bayesian bootstrap.  The best coverage levels only
arose for methods with $\Pr(v_i=0)>0$.  For the larger sample sizes in
this window the power bootstrap $t$ had shorter intervals and nearly
the same coverage as the multinomial bootstrap $t$.

For the very small $n\in[2,10]$, the shortest methods barely make it into
the plotting region with a coverage cutoff near $0.8$.
We also see some non-monotonicity in the coverage levels.

\subsection{Discrete mildly skewed data: $\dpois(1)$}

Data with a discrete distribution can generate numerous
tied values among the $x_i$. The presence of such ties makes 
$\hat\sigma^*=0$ much more common.  This is especially problematic
for methods that also allow $v_i=0$.

Another difficulty arises when the distribution of $x_i$ only
takes integer values, or more generally, only takes values within
some arithmetic sequence in $\real$.
In this case the distribution $F$ does not satisfy a Cram\'er
condition that was a key assumption in the asymptotic
analysis of \cite{hall:1988} and others.  There is thus much
less known about coverage properties even for large $n$.
For count data, with small $n$ and integer counts $w_i$,
it can easily happen that $\bar x^* = \bar x$ and then the
distribution of $t^*$ has an atom at $0$.

Our example for this setting is $x_i\sim\dpois(1)$.  It has
mild skewness and kurtosis both equal to $1$.
Second order accuracy does not hold for such data.

\begin{figure}[t]
\centering
\includegraphics[width=\linewidth]{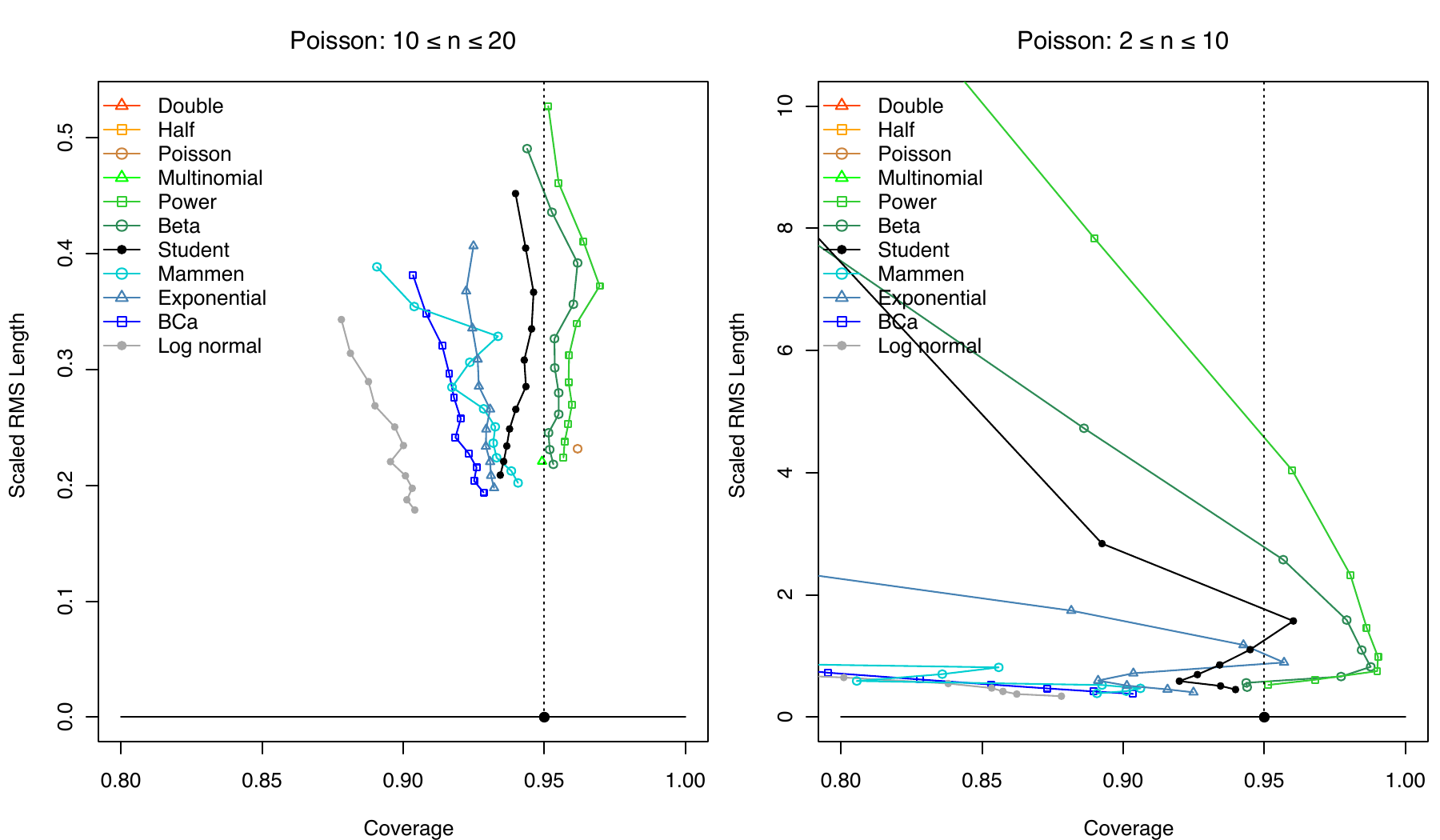}
\caption{\label{fig:spois}
Same as Figure~\ref{fig:sgaus} except that $x_i\sim \dpois(1)$. 
}\end{figure}

Results for Poisson data are shown in Figure~\ref{fig:spois}.
When the top (or bottom) $2.5\%$ of $t^*$ values are
infinite the resulting ACI has infinite length.
All of the double-or-nothing and half sampling ACIs
had infinite RMSL for $2\le n\le 20$. 
For the Poisson bootstrap and the multinomial bootstrap
the first finite RMSL was for $n=20$.  
The clear lesson is that for discrete data, we should avoid
bootstraps with $\Pr(v_i=0)>0$.

For $10\le n\le20$, the beta bootstrap slightly over covered. 
Among the under-covering methods, the Mammen bootstrap $t$
did the best.
For $n\le 10$ the attained coverage levels of several methods are not 
monotone in $n$ although the RMSLs are monotone.

\subsection{Discrete skewed data/geometric data}
We also consider a geometric distribution 
equal to the number of failures seen prior to
the first success in a sequence of Bernoulli
trials with success probability $p$ 
We set $p=1-\exp(-1) \approx 0.632$.
This distribution has skewness $(2-p)/\sqrt(1-p)\doteq 2.25$
and kurtosis $6 + p^2/(1-p)\doteq7.09$.
It is thus a more severe discrete distribution than $\dpois(1)$.

\begin{figure}[t]
\centering
\includegraphics[width=\linewidth]{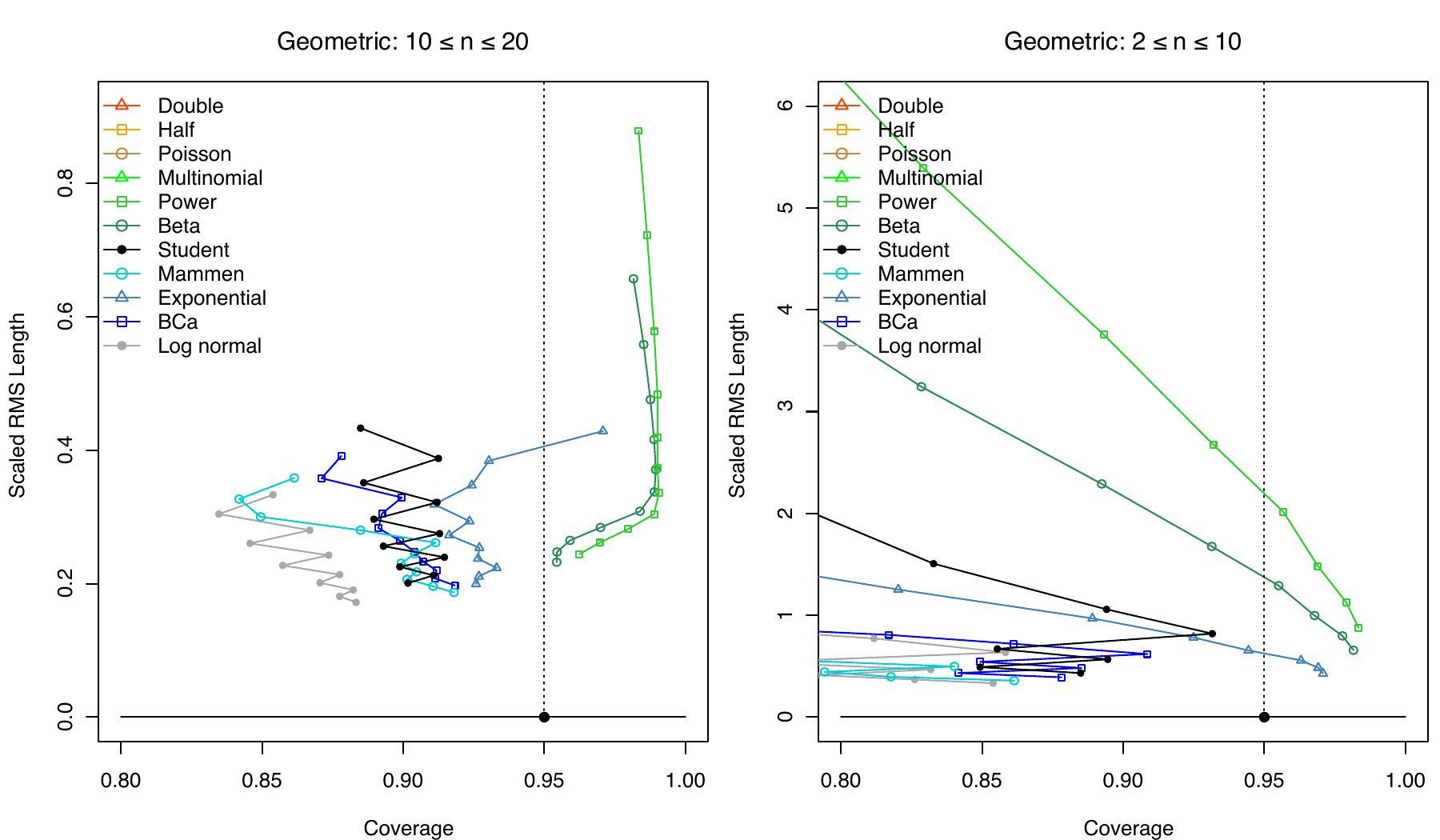}
\caption{\label{fig:sgeom}
Same as Figure~\ref{fig:sgaus} except that 
$\Pr(x_i=x) = p(1-p)^{x-1}$ for $p=1-\exp(-1)$.
}\end{figure}

The results for this distribution are shown in Figure~\ref{fig:sgeom}.
All of the methods with $\Pr(w_i=0)>0$ had
infinite RMSLs for all sample sizes $2\le n\le20$.
For $10\le n\le20$ the Student's $t$ intervals
alternate between higher and lower coverage as $n$
changes between odd and even.  The lognormal
intervals do this too, except that they get shorter
for odd sample sizes while Student's $t$ intervals
get shorter for even sample sizes.
These alternations can be partially understood
in terms of the discreteness of the  distribution
of $(x_1,x_2,\dots,x_n)$.  For example when $n=10$
there is roughly $6.7\%$ probability that the sample
has one $2$, two $1$s and eight $0$s.  It is not clear
why some of the ACIs are strongly affected by 
this discreteness and others are not.  For instance, the
power and beta bootstrap $t$ methods over-cover
the true mean for $7\le n\le 10$ and do not show
alternating coverage.  Similarly, the ACI coverage levels
for Poisson data do not have such a sharp alternation pattern.



\section{Discussion}\label{sec:discussion}

Nonparametric ACIs are well understood in the limit as $n\to\infty$
when $F$ satisfies some conditions based on moments and
also a Cram\'er condition.  This paper looks at small $n$,
at some distributions with large or even infinite kurtosis
as well as distributions with atoms that do not satisfy
the Cram\'er condition.
The conclusions from the numerical investigations are:
\begin{compactenum}[a)]
\item
For general purpose nonparametric ACIs for the
mean of a small sample, we can rule  out \bca\, and the Mammen
and lognormal bootstrap $t$ ACIs.  They all have below nominal coverage
in every example considered here while other methods did better. While \bca\
has advantages they do not extend to ACIs for the mean based on a
small sample.

\item
Student's $t$ intervals did well on symmetric Gaussian and heavy-tailed 
data, but not on skewed data. 
These outcomes are well understood theoretically.

\item
Methods with second order accuracy for one-sided confidence
intervals were not necessarily better at two-sided coverage.
The \bca, Mammen,  and lognormal samplers are second order
accurate but had severe under-coverage.
The second order accurate beta bootstrap $t$ had worse under-coverage
than the power bootstrap $t$ for several of the examples.

\item
Methods with $\Pr(v_i>0)=1$ never give
infinite length confidence intervals. When at least two of the observations are distinct,
they always have $\hat\sigma^*>0$. If all observations are equal
they default to the interval $[\bar x,\bar x]$. 
These methods are especially valuable when the distribution of $x_i$ has atoms.
In those cases the methods with $\Pr(v_i=0)>0$ commonly had
infinite RMSLs.

\item
For the lognormal distribution, which is continuous
with extremely heavy tails, the best coverage came from methods
that could select $v_i=0$.  Of those, the double-or-nothing bootstrap
$t$ had the highest coverage, even over-covering for $n\le 19$.

\item The value $n=10$ is a convenient dividing line between
small and very small sample sizes.  For $n\le 10$ the RMSL
versus coverage curves frequently looked very different from
any expected asymtotic behavior. For $n\ge10$, most curves
already looked to be following a trajectory towards the reference point.
\end{compactenum}

\smallskip

Among  methods with good asymptotic performance
we should prefer those that do well on small sample sizes
because that is where there can be important differences.
Of the methods in this paper the power law bootstrap $t$
and the beta bootstrap $t$ consistently showed
intermediate length and coverage, coming between
methods that severely undercovered or had severely long
intervals. Of those two, only the beta bootstrap $t$ is
second order accurate.

If one has an extreme preference for coverage and
much less concern for interval length, then
the multinomial bootstrap $t$, double-or-nothing bootstrap,
Poisson bootstrap and half sampling bootstrap provide it.
Of those only the Poisson and multinomial bootstrap $t$ samplers
are second order accurate.

We compared \bca, Student's $t$ and several bootstrap $t$
algorithms. The \bca\ algorithm was investigated only
for multinomial weights.  It might benefit from some other
weight distribution.  However, the ACIs from \bca\ are
nested inside the convex hull $[x_{(1)},x_{(n)}]$ of the
data and that already limits their coverage for small $n$.
Empirical likelihood ACIs \citep{elbook} benefit from bootstrap calibration
but are also nested inside the convex hull,
though there has been some work on versions that extend
beyond the convex hull of the data \citep{chen:vari:abra:2008,
emer:owen:2009,tsao:wu:2013}.

\section*{Acknowledgments}

This work was supported by the National Science Foundation under
grant DMS-2152780.  I thank Tilmman Gneiting, Enno Mammen,
Lisa Hofer and Tim Hesterberg for helpful discussions.

\bibliographystyle{apalike}
\bibliography{boot}

\begin{thebibliography}{}

\bibitem[Angelopoulos and Bates, 2023]{ange:bate:2023}
Angelopoulos, A.~N. and Bates, S. (2023).
\newblock Conformal prediction: A gentle introduction.
\newblock {\em Foundations and Trends{\textregistered} in Machine Learning},
  16(4):494--591.

\bibitem[Austern and Mackey, 2022]{aust:mack:2022}
Austern, M. and Mackey, L. (2022).
\newblock Efficient concentration with {Gaussian} approximation.
\newblock Technical report, arXiv:2208.09922.

\bibitem[Bahadur and Savage, 1956]{baha:sava:1956}
Bahadur, R.~R. and Savage, L.~J. (1956).
\newblock The nonexistence of certain statistical procedures in nonparametric
  problems.
\newblock {\em Annals of Mathematical Statistics}, 27(4):1115--1122.

\bibitem[Boos and Hughes-Oliver, 2000]{boos:hugh:2000}
Boos, D.~D. and Hughes-Oliver, J.~M. (2000).
\newblock How large does n have to be for {Z} and t intervals?
\newblock {\em The American Statistician}, 54(2):121--128.

\bibitem[Brehmer and Gneiting, 2021]{breh:gnei:2021}
Brehmer, J.~R. and Gneiting, T. (2021).
\newblock Scoring interval forecasts: Equal-tailed, shortest, and modal
  interval.
\newblock {\em Bernoulli}, 28(3):1993--2010.

\bibitem[Casella et~al., 1994]{case:hwan:robe:1994}
Casella, G., Hwang, J.~G., and Robert, C.~P. (1994).
\newblock Loss functions for set estimation.
\newblock In Gupta, S.~S. and Berger, J.~O., editors, {\em Statistical Decision
  Theory and Related Topics V}, pages 237--251. Springer-Verlag, Inc., New
  York.

\bibitem[Casella and Hwang, 1991]{case:hwan:1991}
Casella, G. and Hwang, J.~T. (1991).
\newblock Evaluating confidence sets using loss functions.
\newblock {\em Statistica Sinica}, pages 159--173.

\bibitem[Chen et~al., 2008]{chen:vari:abra:2008}
Chen, J., Variyath, A.~M., and Abraham, B. (2008).
\newblock Adjusted empirical likelihood and its properties.
\newblock {\em Journal of Computational and Graphical Statistics},
  17(2):426--443.

\bibitem[Cressie, 1980]{cres:1980}
Cressie, N. (1980).
\newblock Relaxing assumptions in the one sample $t$-test.
\newblock {\em Australian Journal of Statistics}, 22(2):143--153.

\bibitem[Cs{\"o}rg{\H{o}} et~al., 2014]{csor:mart:nasa:2024}
Cs{\"o}rg{\H{o}}, M., Martsynyuk, Y.~V., and Nasari, M.~M. (2014).
\newblock Another look at bootstrapping the student t-statistic.
\newblock {\em Mathematical Methods of Statistics}, 23(4):256--278.

\bibitem[Diaconis and Holmes, 1994]{diac:holm:1994}
Diaconis, P. and Holmes, S. (1994).
\newblock Gray codes for randomization procedures.
\newblock {\em Statistics and Computing}, 4:287--302.

\bibitem[Dice, 1945]{dice:1945}
Dice, L.~R. (1945).
\newblock Measures of the amount of ecologic association between species.
\newblock {\em Ecology}, 26(3):297--302.

\bibitem[DiCiccio and Efron, 1996]{dici:efro:1996}
DiCiccio, T.~J. and Efron, B. (1996).
\newblock Bootstrap confidence intervals.
\newblock {\em Statistical Science}, 11(3):189--228.

\bibitem[Donoho, 2024]{dono:2024}
Donoho, D. (2024).
\newblock Data science at the singularity.
\newblock {\em Harvard Data Science Review}, 6(1).

\bibitem[Efron, 1969]{efro:1969}
Efron, B. (1969).
\newblock Student's $t$-test under symmetry conditions.
\newblock {\em Journal of the American Statistical Association},
  64(328):1278--1302.

\bibitem[Efron, 1981]{efro:1981}
Efron, B. (1981).
\newblock Nonparametric standard errors and confidence intervals.
\newblock {\em canadian Journal of Statistics}, 9(2):139--158.

\bibitem[Efron, 1987]{efro:1987}
Efron, B. (1987).
\newblock Better bootstrap confidence intervals.
\newblock {\em Journal of the American statistical Association},
  82(397):171--185.

\bibitem[Efron and Tibshirani, 1993]{efro:tibs:1993}
Efron, B.~M. and Tibshirani, R.~J. (1993).
\newblock {\em An Introduction to the Bootstrap}.
\newblock Chapman and Hall, Boca Raton, FL.

\bibitem[Emerson and Owen, 2009]{emer:owen:2009}
Emerson, S.~C. and Owen, A.~B. (2009).
\newblock Calibration of the empirical likelihood method for a vector mean.
\newblock {\em Electronic Journal of Statistics}, 3:1161--1192.

\bibitem[Fisher and Hall, 1991]{fish:hall:1991}
Fisher, N.~I. and Hall, P. (1991).
\newblock Bootstrap algorithms for small samples.
\newblock {\em Journal of Statistical Planning and Inference}, 27(2):157--169.

\bibitem[Frongillo and Kash, 2014]{fron:kash:2014}
Frongillo, R. and Kash, I. (2014).
\newblock General truthfulness characterizations via convex analysis.
\newblock In {\em Web and Internet Economics: 10th International Conference,
  WINE 2014, Beijing, China, December 14--17, 2014. Proceedings 10}, pages
  354--370. Springer.

\bibitem[Gneiting, 2011]{gnei:2011}
Gneiting, T. (2011).
\newblock Making and evaluating point forecasts.
\newblock {\em Journal of the American Statistical Association},
  106(494):746--762.

\bibitem[Gneiting and Raftery, 2007]{gnei:raft:2007}
Gneiting, T. and Raftery, A.~E. (2007).
\newblock Strictly proper scoring rules, prediction, and estimation.
\newblock {\em Journal of the American Statistical Association},
  102(477):359--378.

\bibitem[Greenland and Poole, 2013]{gree:pool:2013}
Greenland, S. and Poole, C. (2013).
\newblock Living with p values: resurrecting a {Bayesian} perspective on
  frequentist statistics.
\newblock {\em Epidemiology}, 24(1):62--68.

\bibitem[Guillou, 1999]{guil:1999}
Guillou, A. (1999).
\newblock Efficient weighted bootstraps for the mean.
\newblock {\em Journal of Statistical Planning and Inference}, 77(1):11--35.

\bibitem[Hall, 1986]{hall:1986}
Hall, P. (1986).
\newblock On the number of bootstrap simulations required to construct a
  confidence interval.
\newblock {\em The Annals of Statistics}, pages 1453--1462.

\bibitem[Hall, 1988]{hall:1988}
Hall, P. (1988).
\newblock Theoretical comparisons of bootstrap confidence intervals.
\newblock {\em The Annals of Statistics}, 16(3):927--953.

\bibitem[Hall and Mammen, 1994]{hall:mamm:1994}
Hall, P. and Mammen, E. (1994).
\newblock On general resampling algorithms and their performance in
  distribution estimation.
\newblock {\em The Annals of Statistics}, pages 2011--2030.

\bibitem[Hand et~al., 2021]{hand:chri:kiri:2021}
Hand, D.~J., Christen, P., and Kirielle, N. (2021).
\newblock {$F^*$}: an interpretable transformation of the {F}-measure.
\newblock {\em Machine learning}, 110(3):451--456.

\bibitem[Hartigan, 1969]{hart:1969}
Hartigan, J.~A. (1969).
\newblock Using subsample values as typical values.
\newblock {\em Journal of the American Statistical Association},
  64(328):1303--1317.

\bibitem[Hesterberg, 1995]{hest:1995}
Hesterberg, T.~C. (1995).
\newblock Tail-specific linear approximations for efficient bootstrap
  simulations.
\newblock {\em Journal of Computational and Graphical Statistics},
  4(2):113--133.

\bibitem[Hesterberg, 2015]{hest:2015}
Hesterberg, T.~C. (2015).
\newblock What teachers should know about the bootstrap: Resampling in the
  undergraduate statistics curriculum.
\newblock {\em The American Statistician}, 69(4):371--386.

\bibitem[Hofer, 2022]{hofe:2022}
Hofer, L.~J. (2022).
\newblock Interval score for comparison of confidence intervals.
\newblock Master's thesis, University of Zurich.
\newblock \url{https://www.biostat.uzh.ch/master-theses}.

\bibitem[Hofer and Held, 2022]{hofe:held:2022}
Hofer, L.~J. and Held, L. (2022).
\newblock Comparing confidence intervals for a binomial proportion with the
  interval score.
\newblock Technical report, arXiv:2207.03199.

\bibitem[Hoyt and Owen, 2024]{meandimradial}
Hoyt, C. and Owen, A.~B. (2024).
\newblock Mean dimension of radial basis functions.
\newblock {\em SIAM Journal on Numerical Analysis}, 62(3):1191--1211.

\bibitem[Koenker, 2005]{koen:2005}
Koenker, R. (2005).
\newblock {\em Quantile regression}.
\newblock Cambridge University Press, Cambridge.

\bibitem[Kohavi et~al., 2014]{koha:etal:2014}
Kohavi, R., Deng, A., Longbotham, R., and Xu, Y. (2014).
\newblock Seven rules of thumb for web site experimenters.
\newblock In {\em Proceedings of the 20th ACM SIGKDD international conference
  on Knowledge discovery and data mining}, pages 1857--1866.

\bibitem[Kufs, 2010]{kufs:2010}
Kufs, C. (2010).
\newblock
  \url{https://statswithcats.net/2010/07/11/30-samples-standard-suggestion-or-superstition/}.
\newblock Accessed: 2025-05-26.

\bibitem[L'Ecuyer et~al., 2023]{ci4rqmc}
L'Ecuyer, P., Nakayama, M.~K., Owen, A.~B., and Tuffin, B. (2023).
\newblock Confidence intervals for randomized quasi-monte carlo estimators.
\newblock In Corlu, C.~G., Hunter, S.~R., Lam, H., Onggo, B.~S., Shortle, J.,
  and Biller, B., editors, {\em 2023 Winter Simulation Conference (WSC)}, pages
  445--456. IEEE.

\bibitem[Leisch, 2019]{R:bootstrap}
Leisch, F. (2019).
\newblock {\em bootstrap: Functions for the Book ``An Introduction to the
  Bootstrap''}.
\newblock S original and from StatLib and by Rob Tibshirani. R port by
  Friedrich Leisch.

\bibitem[Mammen, 1993]{mamm:1993}
Mammen, E. (1993).
\newblock Bootstrap and wild bootstrap for high dimensional linear models.
\newblock {\em The Annals of Statistics}, 21(1):255--285.

\bibitem[Mase et~al., 2024]{mase:owen:seil:2024}
Mase, M., Owen, A.~B., and Seiler, B.~B. (2024).
\newblock Variable importance without impossible data.
\newblock {\em Annual Review of Statistics and Its Application}, 11.

\bibitem[Mason and Newton, 1992]{maso:newt:1992}
Mason, D.~M. and Newton, M.~A. (1992).
\newblock A rank statistics approach to the consistency of a generaal
  bootstrap.
\newblock {\em The Annals of Statistics}, 20(3):1611--1624.

\bibitem[McCarthy, 1969]{mcca:1969}
McCarthy, P.~J. (1969).
\newblock Pseudo-replication: Half samples.
\newblock {\em Revue de l'Institut International de Statistique},
  37(3):239--264.

\bibitem[Miller, 1986]{mill:1986}
Miller, R.~G. (1986).
\newblock {\em Beyond {ANOVA}: Basics of Applied Statistics}.
\newblock Wiley, New York.

\bibitem[Owen, 1988]{owen:smallci}
Owen, A.~B. (1988).
\newblock Small sample central confidence intervals for the mean.
\newblock Technical Report 302, Stanford University, Department of Statistics.
\newblock \url{https://purl.stanford.edu/mz765np4744}, Accessed 14th August
  2023.

\bibitem[Owen, 1992]{elsmall}
Owen, A.~B. (1992).
\newblock Empirical likelihood and small samples.
\newblock In {\em Computing Science and Statistics}, pages 79--88. Springer.

\bibitem[Owen, 2001]{elbook}
Owen, A.~B. (2001).
\newblock {\em Empirical Likelihood}.
\newblock Chapman \& Hall/CRC, Boca Raton, FL.

\bibitem[Owen, 2009]{owen2009monte}
Owen, A.~B. (2009).
\newblock {Monte Carlo} and {quasi-Monte Carlo} for statistics.
\newblock In L'Ecuyer, P. and Owen, A.~B., editors, {\em Monte Carlo and
  Quasi-Monte Carlo Methods 2008}, pages 3--18. Springer.

\bibitem[Owen, 2024]{err4qmc}
Owen, A.~B. (2024).
\newblock Error estimation for {quasi-Monte Carlo}.
\newblock Technical report, arXiv:2501.00150.

\bibitem[Owen, 2025]{likehall}
Owen, A.~B. (2025).
\newblock Coverage errors for {Student's} $t$ confidence intervals comparable
  to those in {Hall} (1988).
\newblock Technical report, arXiv:2501.07645.

\bibitem[Owen and Eckles, 2012]{owen:eckl:2012}
Owen, A.~B. and Eckles, D. (2012).
\newblock Bootstrapping data arrays of arbitrary order.
\newblock {\em Annals of Applied Statistics}, 6(3):895--927.

\bibitem[Oza and Russell, 2001]{oza:russ:2001}
Oza, N.~C. and Russell, S.~J. (2001).
\newblock Online bagging and boosting.
\newblock In {\em International workshop on artificial intelligence and
  statistics}, pages 229--236. PMLR.

\bibitem[Pan and Owen, 2023]{superpolyone}
Pan, Z. and Owen, A.~B. (2023).
\newblock Super-polynomial accuracy of one dimensional randomized nets using
  the median-of-means.
\newblock {\em Mathematics of Computation}, 92(340):805--837.

\bibitem[Pan and Owen, 2025]{lowskewness}
Pan, Z. and Owen, A.~B. (2025).
\newblock Skewness of a randomized {quasi-Monte Carlo} estimate.
\newblock {\em Journal of Complexity}, page 101956.

\bibitem[Powers, 2020]{powe:2020}
Powers, D. M.~W. (2020).
\newblock Evaluation: from precision, recall and {F}-measure to {ROC},
  informedness, markedness and correlation.
\newblock Technical report, arXiv:2010.16061.

\bibitem[Rubin, 1981]{rubi:1981}
Rubin, D.~B. (1981).
\newblock The {B}ayesian bootstrap.
\newblock {\em The Annals of Statistics}, 9:130--134.

\bibitem[Singh, 1981]{sing:1981}
Singh, K. (1981).
\newblock On the asymptotic accuracy of {Efron's} bootstrap.
\newblock {\em The Annals of Statistics}, 9(6):1187--1195.

\bibitem[Sørensen, 1948]{sore:1948}
Sørensen, T. (1948).
\newblock A method of establishing groups of equal amplitude in plant sociology
  based on similarity of species and its application to analyses of the
  vegetation on {Danish} commons.
\newblock {\em Kongelige Danske Videnskabernes Selskab}, 5(4):1--34.

\bibitem[Tsao and Wu, 2013]{tsao:wu:2013}
Tsao, M. and Wu, F. (2013).
\newblock Empirical likelihood on the full parameter space.
\newblock {\em The Annals of Statistics}, 41(4):2176--2196.

\bibitem[Van~Rijsbergen, 1979]{vanr:1979}
Van~Rijsbergen, C.~J. (1979).
\newblock {\em Information retrieval}.
\newblock Butterworth and Co., London.

\bibitem[Waudby-Smith and Ramdas, 2024]{waud:ramd:2024}
Waudby-Smith, I. and Ramdas, A. (2024).
\newblock Estimating means of bounded random variables by betting.
\newblock {\em Journal of the Royal Statistical Society Series B}, 86(1):1--27.

\bibitem[Winkler, 1972]{wink:1972}
Winkler, R.~L. (1972).
\newblock A decision-theoretic approach to interval estimation.
\newblock {\em Journal of the American Statistical Association},
  67(337):187--191.

\bibitem[Wolfowitz, 1950]{wolf:1950}
Wolfowitz, J. (1950).
\newblock Minimax estimates of the mean of a normal distribution with known
  variance.
\newblock {\em The Annals of Mathematical Statistics}, 21(2):218--230.

\bibitem[Xu et~al., 2020]{xu:etal:2020}
Xu, L., Gotwalt, C., Hong, Y., King, C.~B., and Meeker, W.~Q. (2020).
\newblock Applications of the fractional-random-weight bootstrap.
\newblock {\em The American Statistician}, 74(4):345--358.

\end{thebibliography}
\appendix

\section{Proofs}\label{sec:proofs}

\subsection{Technical lemmas}

\begin{lemma}\label{lem:invmoment}
Let $v_1,\dots,v_n$ be independent with $\e(v_i)$ uniformly bounded
away from $0$ and from infinity.
Suppose also that $\e( v_i^{-\alpha})\le M_\alpha<\infty$ holds
for all $i=1,\dots,n$ and some $\alpha>0$. Let $p<0$ satisfy
$n\ge -p/\alpha$. Then
$$
\e\Bigl( \Bigl(\frac{\sum_{i=1}^n v_i}{\sum_{i=1}^n\e(v_i)}\Bigr)^p\Bigr) \le 
\Bigl( \frac1{\sqrt[\alpha]{M_\alpha}\min_{1\le i\le n}\e(v_i)}\Bigr)^p.
$$
\end{lemma}
\begin{proof}
This is Proposition A.3 of \cite{meandimradial}.
\end{proof}
Our setting has $\e(v_i)=1$ and so
for $\bar v = (1/n)\sum_{i=1}^nv_i$,  we get
$\e\bigl( \bar v^p\bigr) \le M_\alpha^{-p/\alpha}$
when $n\ge |p|/\alpha$.
For the beta bootstrap $t$ of Section~\ref{sec:betabootstrap-t},
we can use any $\alpha\in(0,1/2)$.

\begin{lemma}\label{lem:somepows}
For $n\ge1$, let $v_1,\dots,v_n$ be IID random variables with $\e(v_i^k)=\mu_k$ for $k\ge1$ with $\mu_1=1$ and $\mu_2=2$.
Then
\begin{align}
\e( \olvk(\bar v-1)) &= \frac{\mu_{k+1}-\mu_k}n,\quad\text{and}\label{eq:vkbv}\\
\e( \olvk(\bar v-1)^2) 
&= \frac{\mu_k}n +O\Bigl(\frac1{n^2}\Bigr).\label{eq:vkbv2}
\end{align}
\end{lemma}
\begin{proof}
  First
  $$\e( \olvk\bar v) =\frac1{n^2}\sum_{i=1}^n\sum_{j=1}^n\e( v_i^kv_j) = \mu_k\mu_1+\frac{\mu_{k+1}-\mu_k\mu_1}n
= \mu_k +\frac{\mu_{k+1}-\mu_k}n$$
establishing \eqref{eq:vkbv}. Second
\begin{align*}
\e( \olvk\bar v^2) &=\frac1{n^3}\sum_{i=1}^n\sum_{j=1}^n\sum_{\ell=1}^n
\e( v_i^kv_jv_\ell) \\
&=\frac1{n^3}\bigl( n\mu_{k+2}+n(n-1)(\mu_k\mu_2+2\mu_{k+1}\mu_1)+n(n-1)(n-2)\mu_k\mu_1^2\bigr)\\
&=\mu_k +\frac{2\mu_{k+1}-\mu_k}n+\frac{\mu_{k+2}-2\mu_{k+1}+\mu_k}{n^2}
\end{align*}
after rearrangement.
Now
\begin{align*}
\e( \olvk(\bar v-1)^2) &=
\mu_k +\frac{2\mu_{k+1}-\mu_k}n+\frac{\mu_{k+2}-2\mu_{k+1}+\mu_k}{n^2}
-2\Bigl( \mu_k +\frac{\mu_{k+1}-\mu_k}n\Bigr)+\mu_k\\
&=\frac{\mu_k}n+O\Bigl(\frac1{n^2}\Bigr),
\end{align*}
establishing~\eqref{eq:vkbv2}.
\end{proof}

\subsection{Proof of Theorem~\ref{thm:2ndorder}}\label{proof:thm:2ndorder}
Recall that $W_i = nv_i/\sum_{i'=1}^nv_{i'}$.
We begin the case of $\e(W_i^2)$ by using a Taylor expansion of $1/\bar v^2$ around $\bar v=1$ getting
\begin{align}\label{eq:fork2}
\frac{\olvv}{\bar v^2} = \olvv\bigl( 1-2(\bar v-1)+3(\bar v-1)^2-4(\bar v-1)^3z^{-5}\bigr)
\end{align}
where $z$ is between $1$ and $\bar v$.
The expected sum of the first three terms
using $k=2$ in Lemma~\ref{lem:somepows} is
$$
\mu_2 -2\frac{\mu_3-\mu_2}n+3\frac{\mu_2}n+O\Bigl(\frac1{n^2}\Bigr)
=\mu_2 +\frac{10-2\mu_2}n+O\Bigl(\frac1{n^2}\Bigr).
$$
It remains to bound the error term.

For $z$ between $\bar v$ and $1$,
$1/z \le 1 + 1/\bar v$.  Then $\e( z^{-t})=O(\e(\bar v^{-t}))=O(1)$
for any $t>0$ and
\begin{align*}
\e\bigl( \olvv (\bar v-1)^4z^{-5}\bigr)
\le \e\Bigl( {\olvv}^{\,a}\Bigr)^{1/a}
\e\Bigl( (\bar v-1)^{4b}\Bigr)^{1/b}
\e\Bigl( z^{-5c}\Bigr)^{1/c}
\end{align*}
for positive $a,b,c$ with $1/a+1/b+1/c=1$.  Taking $a=2+\epsilon$,
$b=2$ and $c=2+4/\epsilon$ we find that the remainder is
$$
O(1)O(n^{-2})\e( z^{-5c})^{1/c}
=O(1)O(n^{-2})O(\e(\bar v^{-5c})^{1/c})=O(n^{-2})
$$
as $n\to\infty$, 
using Lemma~\ref{lem:invmoment} on the third factor,
establishing equation~\eqref{eq:mu2}.

For $k=3$, the Taylor expansion is
\begin{align*}
\frac{\olvvv}{\bar v^3}
&= \olvvv\bigl(1 -3(\bar v-1)+6(\bar v-1)^2-10(\bar v-1)^3z^{-6} \bigr)
\end{align*}
where $z$ is between $1$ and $\bar v$.
The expected sum of the first three terms
using $k=3$ in Lemma~\ref{lem:somepows} is
$$
\mu_3 -3\frac{\mu_4-\mu_3}n+6\frac{\mu_3}n+O\Bigl(\frac1{n^2}\Bigr)
=\mu_3 +\frac{9\mu_3-3\mu_4}n+O\Bigl(\frac1{n^2}\Bigr).
$$

The $r=3$ remainder term is at most
$|\e\bigl( 10\olvvv (\bar v-1)^3 z^{-6}\bigr)|$.
We can show that it is $O(n^{-3/2})$ similarly to
the prior case, 
this time using $a=3+\epsilon$, $b=2$ and $c = (3+\epsilon)/(2+\epsilon)$
to establish~\eqref{eq:mu3}.

\subsection{Proof of Theorem~\ref{thm:threedistinct}}
\label{sec:proof:thm:threedistinct}

Let $S\subseteq\{1,2,\dots,n\}$ have $k\ge 3$ distinct values
of $x_i$ and let $\delta$ be the smallest value of $|x_i-x_j|$
for distinct $i,j\in S$.
Now
\begin{align*}
(t^*)^2
&=\frac{n\sum_{i=1}^n\sum_{j=1}^nv_iv_j(x_i-\bar x)(x_j-\bar x)}
{\frac12\sum_{i=1}^n\sum_{j=1}^nv_iv_j(x_i-x_j)^2},
\end{align*}
where we can take $v_i\simiid\dbeta(1/2,3/2)$ as the factor of 4
cancels between numerator and denominator.
The numerator is bounded, and by the AM-GM inequality
the denominator is no smaller than
\begin{align*}
\frac{\delta^2}2\sum_{i\in S}\sum_{j\in S\setminus\{i\}}v_iv_j
\le \frac{k(k-1)\delta^2}2\prod_{i\in S}\prod_{j\in S\setminus\{i\}}(v_iv_j)^{1/k(k-1)}
= \frac{k(k-1)\delta^2}2\prod_{i\in S}v_i^{1/k}.
\end{align*}
Letting $C$ be the upper bound for the numerator
$$
\e( (t^*)^2) \le \frac{2C}{k(k-1)\delta^2}\e( v_1^{-1/k})^k.
$$
Now $\e(v_1^{-1/k}) = \int_0^1 v^{-1/2-1/k}(1-v)^{1/2}/B(1/2,3/2)\rd v <\infty$
for $k\ge3$.
The power bootstrap $t$ replaces $\dbeta(1/2,3/2)$ by
$\dbeta(a,1)$ for $a=\sqrt{2}-1$.  It then suffices to have
$a-1-1/k> -1$, which also holds for $k\ge3$.

\subsection{Proof of Theorem~\ref{thm:pairofthrees}}
\label{sec:proof:thm:pairofthrees}

Relabel the observations if necessary so that $x_i=x_1$ for $1\le i\le n_1$
and $x_i = x_n$ for $n_1<i\le n$.
Letting $\lambda=n_1/n$ and $\lambda^*
=\sum_{i=1}^{n_1}w_i$, we find that
$\bar x^*-\bar x=(\lambda^*-\lambda)(x_1-x_n)$
and $\hat\sigma^{*2} = \lambda^*(1-\lambda^*)(x_1-x_n)^2$.
Then
\begin{align}\label{eq:tstar2}
t^*= \sqrt{n}\frac{\lambda^*-\lambda}{\sqrt{\lambda^*(1-\lambda^*)}}\sign(x_1-x_n).
\end{align}
For $j\in\{1,n\}$ let $u_j=\sum_{i=1}^nv_i1\{x_i=x_j\}$ for $v_i\simiid\dbeta(1/2,3/2)$,
so
\begin{align*}
(t^*)^2 &= n\frac{(\lambda^*-\lambda)^2}{\lambda^*(1-\lambda^*)}
\le \frac{n}{\lambda^*(1-\lambda^*)}=
\frac{n(u_1+u_n)^2}{u_1u_n}
=n\Bigl(\frac1{u_1}+\frac1{u_n}+2\Bigr).
\end{align*}
Finally, 
$\e(1/u_1)<\infty$ if $\int_0^1v^{-1/2-1/n_1}(1-v)^{1/2}\rd v<\infty$,
so $\e( (t^*)^2\giv x_1,\dots,x_n)<\infty$ holds if $\min(n_1,n-n_1)\ge3$.

\section{Literature on small $n$ bootstrap}\label{sec:smallnlit}

While most of the theory for the bootstrap is asymptotic as $n\to\infty$
there have also been some combinatorial investigations of the bootstrap
that pertain to small $n$.

\cite{fish:hall:1991} investigate many small sample issues.
Suppose that all $x_i$ are distinct as must happen when $F$ 
has no atoms.  They note that
$\Pr( \hat\sigma^*=0)=n^{1-n}$. 
Also there are ${2n-1 \choose n}$
distinct multinomial count vectors $\bsv=(v_1,\dots,v_n)$.
The most likely of these is $(1,1,\dots,1)$ with probability $n!/n^{n}$.
The least likely count vectors have only one
nonzero element.  There are $n$ such outcomes with probability $n^{-n}$ each.

Let the sorted values of $v_i$ be $v_{[1]}\ge v_{[2]} \ge \cdots \ge v_{[1]}$.
\cite{fish:hall:1991} take a special interest in outcomes 
where $v_{[1]}$ is $n-1$ or $n-2$ as those can give very small values of $\hat\sigma^*$.
They find that
\begin{align*}
\Pr( v_{[1]} =n-1)&=\Bigl(1-\frac1n\Bigr)n^{3-n}&&\text{for $n\ge3$}\\
\Pr( v_{[1]} =n-2, v_{[2]}=v_{[3]}=1)&=\frac12\Bigl(1-\frac1n\Bigr)^2\Bigl(1-\frac2n\Bigr)n^{5-n}&&\text{for $n\ge4$}\\
\Pr( v_{[1]} =n-2, v_{[2]}=2)&=\frac12\Bigl(1-\frac1n\Bigr)^2n^{4-n}&&\text{for $n\ge5$}.
\end{align*}

They tabulate the probability that some value appears
at least $n-2$ out of $n$ times for $n\ge5$.
They find that it requires $n\ge7$ for this probability
to be below $0.05$.  They then say
``Therefore we may expect to experience problems with
tied resample values when using the percentile-$t$
method in samples of size $n\le 6$.''
Those tied resample values do serve to make the bootstrap $t$
intervals wider than they otherwise would be and hence
increase their coverage.

\cite{fish:hall:1991} adopt a convention for $\hat\sigma^*=0$. 
Translating to the present notation, their equation (2.2)
interprets $t^*=(\bar x^*-\bar x)/\hat\sigma^*\le Q$ as 
\begin{align}\label{eq:fishhall2.2}
\bar x^*-\bar x \le\hat\sigma^*Q.
\end{align}
Now if $\hat\sigma^*=0$ and $\bar x^*<\bar x$
then~\eqref{eq:fishhall2.2} holds for every real $Q$ no
matter how small, so we should take $t^*=-\infty$.
Similarly for $\bar x^*>\bar x$, equation~\eqref{eq:fishhall2.2}
is false no matter how large we make $Q$ and so we should take $t^*=\infty$.
This matches our convention of taking $t^*$ to be $+\infty$ when
the numerator is positive and the denominator is zero and $-\infty$
when we divide a negative number by $0$.

They do not mention a convention for the case where $\bar x^*-\bar x =\hat\sigma^*=0$.
That can only happen when there are ties in the data and they assumed continuously
distributed data. 
They describe ties as arising from rounding which they can then undo by adding a uniform variable
to each $x^*_i$ in each bootstrap sample.  Count data are discrete
and not well described by rounding of continuous data.

\cite{diac:holm:1994} show how to use Gray codes to efficiently enumerate
all of the bootstrap sampling counts $\bsv=(v_1,\dots,v_n)$.  This allows an exact
bootstrap calculation to be done by assigning the appropriate weight to the
statistic computed from each of the ${2n-1\choose n}$ different count vectors,
instead of resampling $B$ times.  In the Gray code ordering any two consecutive
count vectors differ in only two places allowing some efficient updates to be used.

\section{Conventions and edge cases}\label{sec:conventions}

There are some odd cases that arise where we need
to define a convention. These are nuisances that can arise
for small $n$ but have exponentially small probability
of happening as $n$ grows.  We need to handle them
when working with small~$n$.

First, when $\hat\sigma^*=0$ and $\bar x^*\ne\bar x$
then we get $t^*=\pm\infty$.  
Our ACIs will still have a finite lower confidence limit if
$t^*=\infty$ happens fewer than $B\alpha_1$ times
and similarly for the upper limits.  For some methods
the $B\to\infty$ version of the bootstrap $t$ interval
truly does have infinite length. We choose to accept that
as an (undesirable) property of any such method instead of tweaking
the method to avoid that problem.  If the true 
length is infinite for some value of $n$, that shows
that the method is not suitable for such~$n$.

Next, it is possible to have $\bar x^*=\bar x$
at the same time that $\hat\sigma^*=0$.
This can happen when one of the $x_i=\bar x$.
This has positive probability for count data.
Then $t^*=0/0$ and we must choose what to do.
One choice is to deliver $t^*=0$.
A second choice is to split the resampled $t^*$ into two $t^*$'s
one equal to each of $\pm\infty$ getting weight
$1/(2B)$ each when forming quantiles.  This is too
cumbersome and not clearly better, so when
$\bar x^*=\bar x$ and $\hat\sigma^*=0$ we
take $t^*=0$ by convention.

When $\rho=\Pr(v=0)>0$ and the $v_i$ are IID,
we can get $v_1=v_2=\cdots=v_n=0$.
Then there is no good way to define the weights $w_i$.
Our convention is to discard such weight vectors and
keep sampling until we have $B$ pseudo-count
vectors that all have $\sum_{i=1}^nv_i>0$. The resulting $w_i$
are still exchangeable.
An alternative is to take $t^*=0$ for such cases,
as if $w_i$ were all equal, and hence equal to $1/n$.
The probability that this event happens is $\rho^n$,
which equals $\exp(-n)$ for the Poisson bootstrap
and $2^{-n}$ for double-or-nothing.


In multinomial sampling, there is probability $n^{1-n}$ that
all sample indices $j(i)$ are identical.  In that case we will get
$\hat\sigma^*=0$ even when $x_i$ are all distinct.
If $\hat\sigma^*=0$ often enough, then the bootstrap $t$ ACI
can have infinite length.  We can ensure that the true length
of the bootstrap $t$ ACI is finite, for $x_i$ with a continuous distribution, 
by taking $n^{1-n}<\min(\alpha_1,1-\alpha_2)$.
For  $\alpha_1=1-\alpha_2=0.025$ we require
$n\ge4$.  When $n=4$ and $B=2000$, we need $B\times0.025=50$
infinite $t^*$ values to get an infinite sample quantile.
Then $\Pr( \dbin(2000,4^{-3})\ge50)\doteq 0.0011$ so 
$\hat\sigma^*=0$ is moderately rare for $n=4$
and since $\Pr( \dbin(2000,5^{-4})\ge50)\doteq 1.5\times10^{-41}$
is a non-issue for continously distributed data with $n\ge5$.
From here on, we only consider the $B\to\infty$ case.

The above analysis is somewhat conservative because
continuous data will have $x_{(1)} <\bar x < x_{(n)}$ with probability one.
Then the events with $\hat\sigma^*=0$ will be split between 
some that give $t^*=-\infty$ and others that give $t^*=+\infty$
and an infinite interval only arises when one of those sub-populations
has probability $0.025$ or more.
When $x_{(j)}<\bar x < x_{(j+1)}$, then an infinite lower
limit has probability $jn^{-n}$ in multinomial sampling,
and so $\Pr( t^* = -\infty\giv x_1,\dots,x_n) \le (n-1)/n^n$ which is below $0.025$
for $n\ge4$. By symmetry
$\Pr( t^* = \infty\giv x_1,\dots,x_n) <0.025$ also holds for $n\ge4$.

For the double-or-nothing bootstrap, 
$\Pr(Y=1)=n2^{-n}$ and $\Pr(Y=1\giv Y>0)=n2^{-n}/(1-2^{-n})$.
Both of these become below $0.025$ once $n\ge9$.
The same holds for  $(n-1)2^{-n}$ and $(n-1)2^{-n}/(1-2^{-n})$.

For the Poisson bootstrap
$\Pr(Y=1) = n\sum_{x=0}^\infty (e^{-1}/x!)^n$.
With $n\ge7$ and $B=\infty$, $\Pr( |t^*|=\infty\giv x_1,\dots,x_n)=0$
when $x_i$ are distinct.  If we don't discard cases with $Y=0$
then this holds for $n\ge6$.


The \bca\ method can never give us an ACI of infinite length
because those intervals are always subsets of $[x_{(1)},x_{(n)}]$.
The \bca\ intervals are not well defined when $x_1=x_2=\cdots=x_n$.
In that case we take the ACI to be $[\bar x,\bar x]$.
One would not normally seek a confidence interval for the mean
of a sample where all the values were equal. However in a high
throughput setting it is worthwhile to have a default way to handle
such irregular cases.

\end{document}